\documentclass[10pt]{article}

\usepackage{amsmath,amsfonts,amssymb,amsthm,graphicx}
\usepackage{enumerate,environ}
\usepackage{etoolbox}
\usepackage[colorlinks=true,citecolor=blue,pdfpagemode=UseNone,pdfstartview=FitH]{hyperref}

\usepackage[shortcuts]{extdash}

\allowdisplaybreaks[4]

\newcommand{\noop}[1]{}

\newcommand{\N}{\mathbb{N}}

\newcommand{\m}{\mathbf{m}}
\newcommand{\Chaitin}{\mathbf{\Omega}}
\newcommand{\BB}{\mathrm{BB}}

\newcommand{\dL}{d_{\mathrm{e}}}
\newcommand{\dML}{d_{\mathrm{p}}}

\newcommand{\set}{^\textnormal{set}}

\newcommand{\st}{\mid}

\newcommand{\LLL}{\mathcal{L}}

\newcommand{\NNN}{\mathcal{N}}
\newcommand{\SSS}{\mathcal{S}}
\newcommand{\XXX}{\mathcal{X}}

\newcounter{cnumber}
\newcommand{\cnew}{\stepcounter{cnumber}c_{\arabic{cnumber}}}
\newcommand{\cold}{c_{\arabic{cnumber}}}
\newcommand{\colder}{c_{\the\numexpr\value{cnumber}-1\relax}}
\newcommand{\coldest}{c_{\the\numexpr\value{cnumber}-2\relax}}
\newcommand{\clabel}[1]{%
  \ifcsname c@#1\endcsname
    \ifnum\value{#1}=\value{cnumber}\else
      \PackageWarning{clabel}{Multiply-defined constant label `#1'}%
    \fi
  \else
    \newcounter{#1}%
    \AtEndDocument{\ifcsname cused@#1\endcsname\else
      \PackageWarningNoLine{clabel}{Unused constant label `#1'}\fi}%
  \fi
  \setcounter{#1}{\value{cnumber}}%
}
\newcommand{\cref}[1]{\cused{#1}c_{\arabic{#1}}}
\protected\def\cused#1{\expandafter\gdef\csname cused@#1\endcsname{}}

\newenvironment{FULLequation}[1]{\begin{equation*}}{\end{equation*}}

\theoremstyle{plain}
\newtheorem{theorem}{Theorem}
\newtheorem{corollary}[theorem]{Corollary}
\newtheorem{lemma}[theorem]{Lemma}
\newtheorem{proposition}[theorem]{Proposition}
\theoremstyle{definition}

\theoremstyle{remark}
\newtheorem{remark}[theorem]{Remark}

\NewEnviron{wideformula}{%
\begin{multline}\BODY\end{multline}}
\newcommand{\NarrowLineBreak}{\\}

\newcommand{\tit}%
  {The universal measure of nonstochastic objects}
\newcommand{\abstr}%
  {The usual settings for discussing the notion of stochastic objects
  are very different from Kolmogorov's original structure-free setting,
  in which there is no predefined external structure on the objects
  (such as the lengths of binary strings).
  A natural modification of stochasticity,
  obtained by replacing finite sets as models
  with probability measures,
  is quasi-stochasticity.
  The structure-free setting is common in studies of quasi-stochasticity,
  and known results show that the universal machine generates non-quasi-stochastic objects
  with exponentially small probability.
  This paper shows that, on the other hand,
  the universal machine generates nonstochastic objects
  surprisingly often.
  It also systematically explores
  numerous connections between stochasticity and quasi-stochasticity
  in the structure-free setting.
  In particular, the definition of quasi-stochasticity can be modified
  in a natural way modelled on Kolmogorov's definition
  so that non-quasi-stochastic objects are produced
  as often as nonstochastic ones;
  the results for this modification imply all the results of this paper
  showing that nonstochastic objects are frequent.%
}

\begin{document}

\title{\tit}

\author{Vladimir Vovk}
\date{October 23, 2026}
\maketitle
\begin{abstract}
  \smallskip
  \abstr

  The version of this paper at \url{http://gtfp.net} (Working Paper 70)
  is updated most often.
\end{abstract}

\section{Introduction}
\label{sec:intro}

The notion of stochastic objects was first introduced
(in the currently recorded history)
in Kolmogorov's talk at a seminar at Moscow State University
on 26 November 1981.
In the same talk he also posed the problem of estimating
the abundance of stochastic and nonstochastic objects.
(Alexander Shen's notes from that seminar have been published
as \cite[note~5.12]{Semenov/etal:2024-full}.)
Various solutions to this problem have been given by
Shen in 1983 \cite{Shen:1983},
Vladimir V'yugin in 1985 \cite{Vyugin:1985},
and Nikolai Vereshchagin and Shen \cite[Proposition~22]{Vereshchagin/Shen:2017}.
See \cite[Sect.~14.2]{Shen/etal:2017book} and \cite[Sect.~2]{Vereshchagin/Shen:2017}
for reviews.

A surprising feature of the existing literature on stochastic objects
is that it always assumes some externally given nested structure on the objects
whose stochasticity we are interested in.\footnote{%
  In this paper I will follow Shen \cite{Shen:1983}
  and G\'acs et al.\ \cite{Gacs/etal:2001}
  and make a terminological distinction between stochastic objects,
  which are defined using finite sets as models,
  and quasi-stochastic objects,
  which are defined using probability measures as models.}
In many cases the objects are the binary strings,
with their lengths playing an important role.
In other cases (e.g., in \cite{Shen:1983}),
the objects are the nonnegative integers with the nested structure
consisting of the sets $\{0,\dots,2^n-1\}$, $n=0,1,\dots$.
Kolmogorov's original problem \cite[note~5.12]{Semenov/etal:2024-full}
is completely different:
no externally given nested structure is used
in measuring sets of nonstochastic objects.
Kolmogorov's setting, discussed in Appendix~\ref{app:Kolmogorov},
still has an intrinsic nested structure
(the sets of objects of complexity at most $k$),
which the key results of this paper dispense with.

The standard problem of measuring the set of all $(\alpha,\beta)$-nonstochastic objects
(where $\alpha$ and $\beta$ are parameters, to be defined in Sect.~\ref{sec:main})
has an important special case:
do such objects even exist in a given level of the nested structure?
(An example is the question of the existence of $(\alpha,\beta)$-nonstochastic strings
of length $n$.)
Important results in this direction are due to Shen \cite[Theorem~2]{Shen:1983}
and G\'acs et al.\ \cite[Theorem~IV.2]{Gacs/etal:2001}.
Such problems do not arise in the structure-free setting of this paper.

Different authors and publications use different ways
of measuring sets of nonstochastic objects.
Shen \cite[Theorem~3]{Shen:1983} bounds the uniform probability measure
of nonstochastic objects.
Using the uniform probability measure
for measuring the amount of nonstochastic objects
appears somewhat unnatural,
and \cite[Theorem~4]{Shen:1983} covers non-uniform probability measures;
Shen's Theorem~4, however, only gives a one-sided bound.
V'yugin \cite[Corollary and Theorem~3]{Vyugin:1985} bounds the universal semimeasure
of nonstochastic objects
(in our structure-free setting the universal semimeasure is a measure,
but this is not necessarily true in structured settings).
The existing result that is perhaps closest to the results of this paper about stochasticity
is Theorem~251 in \cite[Sect.~14.2]{Shen/etal:2017book};
this result uses the same version of the universal semimeasure as this paper
(different from the one used by V'yugin).
In the last two of these results,
those in \cite{Vyugin:1985,Shen/etal:2017book},
and in a plethora of related results in the literature,
the bounds on the universal semimeasure of the $(\alpha,\beta)$-nonstochastic objects
are governed by the term $2^{-\alpha}$.
Kolmogorov's setting described in Appendix~\ref{app:Kolmogorov} is very different,
but we still have a similar prominent role for $2^{-\alpha}$.
The nature of the first key result of this paper,
Theorem~\ref{thm:main},
is even more different, and $2^{-\alpha}$ disappears.

One interpretation of the existing results
on the universal semimeasure of nonstochastic objects
is that ``the probability of obtaining a non-stochastic object
in a random process is negligible''
\cite[Sect.~4.6]{Vereshchagin/Shen:2017}.
According to this paper's results,
whether this is true in the structure-free setting depends
on the details of the definition of stochasticity
(and happens to be true for Vereshchagin and Shen's official definition
\cite[Definition~1]{Vereshchagin/Shen:2017}).
If the definition uses finite sets as models,
the situation in the structure-free setting is the opposite,
and nonstochastic objects are generated extremely often
by the universal semimeasure (Theorem~\ref{thm:main} below).
But if the definition uses probability measures as models,
their counterparts become rare under the standard definitions
(Theorem~\ref{thm:quasi} below).
In the terminology of \cite{Shen:1983,Gacs/etal:2001} and this paper,
nonstochastic objects are frequent while non-quasi-stochastic objects are rare.

The notions of stochasticity and quasi-stochasticity
are regarded as close in the standard structured settings
(see, e.g., \cite[Theorem~1]{Shen:1983} and \cite[Proposition~2]{Vereshchagin/Shen:2017}).
Therefore, nowadays the same term, ``stochasticity'', is often used
for both stochasticity and quasi-stochasticity.
In the structure-free setting a chasm opens up
between stochasticity and quasi-stochasticity.

The first key result of this paper (Theorem~\ref{thm:main})
is stated in Sect.~\ref{sec:main};
it evaluates the likelihood of the universal machine
producing nonstochastic objects.
The result is very simple and implies that nonstochastic objects
are produced very often.
Another result in that section, Theorem~\ref{thm:high},
generalizes Theorem~\ref{thm:main}
and shows that the conditional universal measure
of $(\alpha,\beta)$-nonstochastic objects among objects of high complexity
is bounded away from zero;
see~\eqref{eq:frequent}.

In the case of quasi-stochastic objects
some interesting structure-free results exist in the literature,
and they are in stark contrast with the results of this paper for stochastic objects.
The most explicit ones originate from Epstein and Levin's work
\cite[Proposition~1]{Epstein/Levin:2011};
for proofs, see \cite[Proposition~1]{Levin:2016} and \cite[Lemma~10]{Epstein:2021}.
(And the idea behind those results may go back as far as 1970
\cite[the end of Sect.~1]{Bienvenu/etal:2022b}.)
The report by Bauwens \cite{Bauwens:2010},
posted slightly before Epstein and Levin's report \cite{Epstein/Levin:2011},
is somewhat less explicit.
Results of this kind,
demonstrating the paucity of non-quasi-stochastic objects,
will be discussed in Sect.~\ref{sec:quasi}.

The reason for the contrast between the abundance of nonstochastic objects
and the paucity of non-quasi-stochastic objects
is that the complexity of stochastic objects is simple (i.e., of low complexity),
while this is not true for quasi-stochastic objects.
The simplicity of complexity of stochastic objects
was noted by G\'acs et al.\ \cite[Lemmas~III.8 and~III.10]{Gacs/etal:2001},
and in this paper it is replaced
by the closeness of the prefix and plain complexities for stochastic objects
(and further generalized in Sect.~\ref{sec:stronger}).
Our second key result, Theorem~\ref{thm:quasi} in Sect.~\ref{sec:quasi},
is a version of Epstein and Levin's results
that formalizes the paucity of non-quasi-stochastic objects
in a setting close to Kolmogorov's.

Stronger versions of Theorem~\ref{thm:main} and complementary results
are established in Sect.~\ref{sec:stronger}.
Section~\ref{sec:eight} gives a more systematic picture
covering two complexities (prefix and plain)
and four randomness deficiencies
(based on prefix and plain complexities and on universal e-values and p-values,
to be defined there).
I will state two omnibus results,
one covering quasi-stochasticity (Theorem~\ref{thm:eight})
and the other covering stochasticity (Theorem~\ref{thm:eight-set}).
These omnibus theorems contain the two key results mentioned earlier,
the first one in a stronger form.

All nontrivial proofs for the statements in the main part of the paper
are collected in Appendix~\ref{app:proofs}.
At the beginning of Appendix~\ref{app:proofs}
I will also briefly discuss the main techniques used in this paper.

Kolmogorov's original questions asked in 1981
concern the number of nonstochastic objects
of complexity not exceeding a given level $k$.
His specific questions are answered in Appendix~\ref{app:Kolmogorov}.

Let $\N:=\{1,2,\dots\}$ be the natural numbers without 0,
$\N_0:=\N\cup\{0\}$ be the natural numbers including 0,
$[m]:=\{1,\dots,m\}$ for each $m\in\N$,
$\log$ be binary logarithm,
$C$ be plain (Kolmogorov) complexity
\cite[Sect.~1.1]{Shen/etal:2017book},
$K$ be prefix complexity
\cite[Sect.~4.3]{Shen/etal:2017book},
and $\m$ be a fixed universal semimeasure on $\N$
\cite[Sect.~4.2]{Shen/etal:2017book}.
For a subset $A\subseteq\N$,
we set $\m(A):=\sum_{a\in A}\m(a)$
(therefore, $\m(\{a\})=\m(a)$ for all $a\in\N$).
This makes $\m$ a measure, which is why we refer to it,
above and below, as the \emph{universal measure}.
The words ``increasing'' and ``decreasing'' are used in the wide sense
allowing intervals of constancy.
A ``binary string'' means an element of $\{0,1\}^*$.
For $x\in\{0,1\}^*$,
$\left|x\right|$ is the length of $x$;
$\Box$ is the empty string, $\left|\Box\right|=0$.

In this paper we take $\XXX$ to be our basic space of finite objects
whose $(\alpha,\beta)$-stochasticity we are interested in.
In all formal statements $\XXX$ can be replaced by $\N$
(but in informal discussion using $\N$ in place of $\XXX$
can cause confusion,
since natural numbers play many other roles in this paper).
The basic space $\XXX$ can be any set in computable bijection with $\N$,
such as the binary strings $\{0,1\}^*$.
We will often refer to elements of the basic space $\XXX$ as objects
and will usually denote them by $\omega$.
The complexities $C(\omega)$ and $K(\omega)$
and the universal measure $\m(A)$ for $A\subseteq\XXX$
are defined, naturally, as the corresponding functions ($C$, $K$, or $\m$)
applied to the images of $\omega$ or $A$ under the computable bijection.
When we talk about a \emph{shortest} plain or self-delimiting description
of an object $\omega$,
we mean descriptions in the sense of the optimal decoders
that determine $C$ or $K$, respectively;
therefore, shortest descriptions of $\omega$ have length
$C(\omega)$ for plain descriptions and
$K(\omega)$ for self-delimiting ones.
When convenient, we identify $\XXX$ with $\N$ via the computable bijection,
which gives a meaning to expressions such as $[m]\subseteq\XXX$
and $\max\Omega$ for finite $\Omega\subseteq\XXX$;
in the appendices, we always do so.
For concreteness, the reader can interpret $\XXX$
as $\N$ or $\{0,1\}^*$ in the main part of the paper.

For a brief summary of important results about finitary randomness,
including connections with universal p-values and e-values,
see \cite[Appendix~1]{Vovk:2025ML};
that summary only covers randomness w.r.t.\ uniform measures on finite sets
(and so is relevant to stochasticity but not quasi-stochasticity).

\section{The universal measure of nonstochastic objects}
\label{sec:main}

We say that an object $\omega\in\XXX$ is \emph{$(\alpha,\beta)$-stochastic},
where $\alpha$ is a nonnegative integer and $\beta\in[0,\infty]$,
if there exists a finite set $\Omega\subseteq\XXX$ for which
\begin{itemize}
\item
  $\omega\in\Omega$;
\item
  $K(\Omega)\le\alpha$;
\item
  $K(\omega\mid\Omega) \ge \log\left|\Omega\right| - \beta$.
\end{itemize}
Otherwise $\omega$ is \emph{$(\alpha,\beta)$-nonstochastic}.
In our notation we follow Kolmogorov
(see Appendix~\ref{app:Kolmogorov}).

Let $\NNN_{\alpha,\beta}\subseteq\XXX$ be the set of all $(\alpha,\beta)$-nonstochastic objects.
We are interested in the asymptotics of $\m(\NNN_{\alpha,\beta})$ as $\alpha,\beta\to\infty$.
Let
\[
  \tau(t)
  :=
  \m(\{n\in\N\st n>t\})
\]
be the tail of $\m$; set $\tau(\infty):=0$.

\begin{theorem}\label{thm:main}
  There exists $c>1$ such that,
  for all $\alpha\in\N_0$ and $\beta\le2^{2^{\alpha}}$,
  \begin{equation}\label{eq:main}
    \tau(\alpha)/c
    \le
    \m(\NNN_{\alpha,\beta})
    \le
    c\tau(\alpha).
  \end{equation}
\end{theorem}

The condition $\beta\le 2^{2^{\alpha}}$ in the statement of the theorem 
can be replaced by $\beta\le g(\alpha)$ for a fixed computable function $g$
(``fixed'' in the sense of $c$ depending on $g$).
Moreover, this condition is only needed for the lower bound.
An even stronger version of Theorem~\ref{thm:main} is
\begin{equation}\label{eq:super}
  \tau(\max(\alpha,\beta))/c
  \le
  \m(\NNN_{\alpha,\beta})
  \le
  c\tau(\alpha);
\end{equation}
cf.\ Lemma~\ref{lem:g} below.
In Sect.~\ref{app:main} we will prove Theorem~\ref{thm:main}
in the form \eqref{eq:super}.
The proof will show that the theorem
continues to hold with $\NNN_{\alpha,\beta}$
replaced by its modification $\NNN^C_{\alpha,\beta}$
that has $C$ in place of both occurrences of $K$
in the definition of $\NNN_{\alpha,\beta}$.

The theorem says that
$\m(\NNN_{\alpha,\beta})=\Theta(\tau(\alpha))$
under mild conditions on $\beta$.
Informally, nonstochastic objects are generated with probability
that decays extremely slowly.
The function $n\mapsto\tau(n)$ monotonically decreases to 0
but does so more slowly than any computable function monotonically decreasing to 0,
such as $1/\log\log n$ (and we can take any number of logarithms).
Let us state this well-known observation formally.

\begin{lemma}
  Let $f:\N\to(0,\infty)$ be a computable decreasing function
  such that $f(n)\to0$ as $n\to\infty$.
  Then $\tau(n)/f(n)\to\infty$ as $n\to\infty$.
\end{lemma}

\begin{proof}
  Redefine $f:=\min(f,1)$.
  Set $h(t):=\surd f(t)$ for $t\in\N$ and $h(0):=1$.
  Since $f$ is decreasing and $f(t)\to0$,
  $P(n):=h(n-1)-h(n)$, $n\in\N$,
  is a computable probability measure on $\N$ with tails
  $P(\{n\st n>t\})=h(t)$ for all $t\ge0$.
  The universality of $\m$ gives $\m\ge c P$ for some $c>0$,
  and so, as $t\to\infty$,
  \[
    \frac{\tau(t)}{f(t)}
    \ge
    \frac{c h(t)}{f(t)}
    =
    \frac{c}{\surd f(t)}
    \to
    \infty.
    \qedhere
  \]
\end{proof}

Let us discuss in greater detail what appears to be the main difference
between Kolmogorov's approach to stochasticity
and the standard one relying on a given structure on the objects.
Kolmogorov's definitions and questions,
described in detail in Appendix~\ref{app:Kolmogorov},
are structure-free
in that they only involve the complexities of the objects
(the complexities can be, e.g., plain or prefix,
conditional or unconditional)
and other objects derived from the given objects
(such as finite sets composed of the given objects).
Complexities can be replaced by randomness deficiencies
(which do not have to be defined via complexities);
the point is that structure-free statements
should involve only functions that are invariant,
to within an additive constant,
w.r.t.\ computable permutations of the objects.
An example of a non-invariant function
is provided by the length of a binary string.
In this paper we will be interested only
in structure-free statements.
Theorem~\ref{thm:main} is an example,
and we will see numerous further ones later.

Now let us briefly discuss basic properties of $\tau$
and relations between $\tau$ and other well-known functions.
The following lemma shows how insensitive $\tau$ is to the value of its argument,
and it also shows that \eqref{eq:super} is indeed
a stronger version of Theorem~\ref{thm:main}.

\begin{lemma}\label{lem:g}
  Let $g:\N_0\to\N_0$ be a computable increasing function.
  There is $c_g>0$ such that
  $\tau(g(t))\ge c_g\tau(t)$ for all $t\in\N_0$.
\end{lemma}

Set
\begin{equation*}
  \sigma(t)
  :=
  \m(\{n\st K(n)>t\});
\end{equation*}
we will need this function for the interpretation of Theorem~\ref{thm:high}
later in this section.
The following relation is well known,
but I have not seen it stated explicitly in print.

\begin{lemma}\label{lem:sigma-tau}
  For some constant $c>1$,
  \begin{equation}\label{eq:tau-sigma}
    \tau/c \le \sigma \le c\tau.
  \end{equation}
\end{lemma}

Another interesting connection is with Chaitin's
\cite[Definition~5.3]{Chaitin:1974}
version
\begin{equation}\label{eq:BB}
  \BB(s)
  :=
  \max\{n\in\N\st K(n)\le s\}
\end{equation}
(with $\max\emptyset:=0$)
of the busy beaver function
(see also \cite[Sect.~1.2.2]{Shen/etal:2017book}).
Let us set
\begin{equation}\label{eq:mu}
  \mu(t)
  :=
  \min\{K(n)\st n>t\},
\end{equation}
so that $\mu$ is, essentially, the inverse to the busy beaver function.
Andreev \cite[Theorem~1]{Andreev:2017} proves that
\begin{equation}\label{eq:Andreev}
  \mu-O(\log\mu)
  \le
  -\log\tau
  \le
  \mu+O(1).
\end{equation}
(Andreev's statement uses inverse functions to $\mu$ and $-\log\tau$,
and in \cite[Theorem~2]{Andreev:2017} he shows the optimality of \eqref{eq:Andreev},
crediting G\'acs \cite{Gacs:1983} with part of his Theorem~2.)

Lemmas~\ref{lem:g} and~\ref{lem:sigma-tau}
will be proved in Sect.~\ref{app:tail}.

The condition on $\beta$ in Theorem~\ref{thm:main}
is weak, but it cannot be dropped.
The following theorem describes the dependence
of $-\log\m(\NNN_{\alpha,\beta})$
on $\beta$ completely,
albeit only with logarithmic precision in $\alpha$.

\begin{theorem}\label{thm:beta}
  There is a constant $c>0$ such that,
  for all $\alpha\in\N_0$ and $\beta\in[0,\infty)$,
  \begin{FULLequation}{eq:beta}
    \bigl|
      -\log\m(\NNN_{\alpha,\beta})
      -
      \min(\alpha,\mu(\beta))
    \bigr|
    \le
    c\log(\alpha+2).
  \end{FULLequation}
\end{theorem}

Since $\mu(\beta)>\alpha$ if and only if $\beta\ge\BB(\alpha)$,
Theorem~\ref{thm:beta} says that
\[
  \m(\NNN_{\alpha,\beta})
  =
  \begin{cases}
    2^{-\mu(\beta)\pm O(\log(\alpha+2))} & \text{if $\beta<\BB(\alpha)$},\\
    2^{-\alpha\pm O(\log(\alpha+2))} & \text{otherwise}.
  \end{cases}
\]
Therefore,
for every computable bound $\beta\le g(\alpha)$,
$\m(\NNN_{\alpha,\beta})$ is only polynomially small in $\alpha$
(and Theorem~\ref{thm:main} gives the more precise $\Theta(\tau(\alpha))$).
As $\beta$ grows beyond all computable bounds,
$\m(\NNN_{\alpha,\beta})$ decreases as $2^{-\mu(\beta)}$,
and it settles at the order $2^{-\alpha}$,
familiar from the results discussed in the introduction,
when $\beta$ reaches $\BB(\alpha)$.
See the end of Sect.~\ref{app:stronger} for the proof of Theorem~\ref{thm:beta}.
Another expression for $\m(\NNN_{\alpha,\beta})$ working for all $\beta$
will be given as \eqref{eq:alternative} below.

\subsection*{The conditional universal measure of nonstochastic objects}

The next theorem estimates the universal measure
of nonstochastic objects above a given complexity level.
Requiring $K(\omega)>k$ in addition to nonstochasticity
simply replaces $\alpha$ in the bounds by $\max(\alpha,k)$,
with no condition on $k$ whatsoever.
Theorem~\ref{thm:high} below therefore contains Theorem~\ref{thm:main},
which is its special case for $k:=0$.

\begin{theorem}\label{thm:high}
  There is a constant $c>1$ such that,
  for all $\alpha\in\N_0$, all $\beta\le2^{2^{\alpha}}$, and all $k\in\N_0$,
  \begin{equation}\label{eq:high}
    \tau(\max(\alpha,k))/c
    \le
    \m\bigl(\NNN_{\alpha,\beta}\cap\{\omega\st K(\omega)>k\}\bigr)
    \le
    c\tau(\max(\alpha,k)).
  \end{equation}
\end{theorem}

\noindent
Only the lower bound in \eqref{eq:high} uses a condition on $\beta$.
See Sect.~\ref{app:high} for the proof of Theorem~\ref{thm:high}.

We can see from~\eqref{eq:high} that $\alpha$ and $k$ do not interact,
and the larger of the two wins outright.
Suppose $k\ge\alpha$
(the case $k<\alpha$ is almost vacuous,
since every $(\alpha,\beta)$-nonstochastic object
satisfies $K(\omega)\ge\alpha-O(1)$).
Comparing \eqref{eq:high} with $\sigma(k)=\Theta(\tau(k))$
(see \eqref{eq:tau-sigma})
shows that a significant part of the mass
that $\m$ assigns to the objects of complexity exceeding $k$
is assigned to nonstochastic ones.
Theorem~\ref{thm:main} says that nonstochastic objects
are generated with a probability that decays very slowly;
Theorem~\ref{thm:high} adds that, above the threshold $\alpha$,
nonstochastic objects are not merely plentiful:
above any level $k\ge\alpha$ they carry a constant fraction of the universal measure.
Formally, Theorem~\ref{thm:high} implies that there is a constant $c>0$ such that,
for all $k\in\N_0$, all $\alpha\le2^{2^k}$, and all $\beta\le2^{2^{\alpha}}$,
\begin{equation}\label{eq:frequent}
  \m(\NNN_{\alpha,\beta}\mid K>k)
  \ge
  c,
\end{equation}
where, as usual, $\m(A\mid B):=\m(A\cap B)/\m(B)$
and $K>k$ stands for $\{\omega\mid K(\omega)>k\}$.

\section{Probability measures as models}
\label{sec:quasi}

A popular alternative to Kolmogorov's definition of stochastic objects
uses probability measures in place of finite sets as models
(see, e.g., \cite{Shen:1983} or \cite[Sect.~V]{Gacs/etal:2001}).
For binary strings of length $n$,
the two classes of models lead to the same notion of stochasticity
to within $O(\log n)$
(\cite[Theorem~1]{Shen:1983}, \cite[Proposition~I.2]{Gacs/etal:2001}).
In our structure-free setting there is no $n$,
and changing the class of models has a dramatic effect
on the universal measure of the nonstochastic objects.
First we consider the simplest probability measures:
a probability measure $Q$ on $\XXX$ is \emph{elementary}
if it has a finite support and takes only rational values,
as in \cite{Shen:1983}.
(Much of this discussion also applies
to general computable probability measures as models,
as in \cite[Sect.~V]{Gacs/etal:2001};
they are briefly covered at the end of this section.)
Elementary probability measures are finite objects,
and so $K(Q)$ and $K(\omega\mid Q)$ are defined.

An object $\omega\in\XXX$ is \emph{$(\alpha,\beta)$-quasi-stochastic},
for $\alpha\in\N_0$ and $\beta\in[0,\infty]$,
if there exists an elementary probability measure $Q$ such that
\begin{itemize}
\item
  $K(Q)\le\alpha$;
\item
  $Q(\{\omega\})>0$ and $K(\omega\mid Q)\ge-\log Q(\{\omega\})-\beta$.
\end{itemize}
Let $\NNN^{\mathrm{q}}_{\alpha,\beta}\subseteq\XXX$
be the set of all objects that are not $(\alpha,\beta)$-quasi-stochastic
(or \emph{$(\alpha,\beta)$-non-quasi-stochastic}).
Regarding a finite set $\Omega$ as the uniform probability measure on $\Omega$,
we can see that $\NNN^{\mathrm{q}}_{\alpha+c,\beta+c}\subseteq\NNN_{\alpha,\beta}$
for some constant $c$.
The following version of Epstein and Levin's results (see below)
shows that, in terms of the universal measure,
$\NNN^{\mathrm{q}}_{\alpha,\beta}$ is much smaller than $\NNN_{\alpha,\beta}$.

\begin{theorem}\label{thm:quasi}
  There exists $c>0$ such that, for all $\alpha\in\N_0$ and $\beta\in[0,\infty)$,
  \[
    2^{-\alpha}/c
    \le
    \m(\NNN^{\mathrm{q}}_{\alpha,\beta})
    \le
    c
    \left(
      2^{-\alpha+K(\alpha)}
      +
      2^{-2^{\beta/2-c}}
    \right).
  \]
\end{theorem}

There is a constant $c'$ such that
$2^{-2^{\beta/2-c}}\le2^{-\alpha}$ whenever $\beta\ge2\log(\alpha+1)+c'$.
In this range of~$\beta$,
\begin{equation}\label{eq:quasi-main}
  \m(\NNN^{\mathrm{q}}_{\alpha,\beta})
  =
  O(2^{-\alpha+K(\alpha)}),
\end{equation}
and so the contrast between the two classes of models is stark:
\begin{itemize}
\item
  with probability measures as models,
  non-quasi-stochastic objects are generated by the universal machine
  with an exponentially small probability;
  the term $2^{-\alpha}$,
  familiar from the results discussed in the introduction,
  reappears;
\item
  by Theorem~\ref{thm:main},
  $\m(\NNN_{\alpha,\beta})=\Theta(\tau(\alpha))$ (for $\beta\le2^{2^{\alpha}}$) is only polynomially small in $\alpha$,
  and so nonstochastic objects are generated by the universal machine
  with a probability that tends to 0 very slowly.
\end{itemize}

See Sects.~\ref{app:quasi} and~\ref{app:Q}
for the proofs of the results in this section.
The asymptotic relation \eqref{eq:quasi-main}
is a version of Epstein and Levin's results
first posted in \cite{Epstein/Levin:2011}
(see \cite[Lemma~10]{Epstein:2021} and \cite[Proposition~1]{Levin:2016}
for published proofs).
As mentioned in Sect.~\ref{sec:intro},
those results might go back to Levin's idea
explained in one of his notes to Kolmogorov
likely written in 1970 \cite[Sect.~1]{Bienvenu/etal:2022b}.
One of their differences from Theorem~\ref{thm:quasi}
is that Epstein and Levin merge the two standard parameters, $\alpha$ and $\beta$, into one
and concentrate on a single number
as their measure of stochasticity for an object $\omega$,
such as
\[
  \Lambda(\omega)
  :=
  \min_Q
  \bigl(
    K(Q)
    +
    3\log\max(\dL(\omega\mid Q),1)
  \bigr)
\]
\cite[Definition~1]{Epstein:2021},
where $Q$ ranges over the elementary probability measures
and $\dL$ is defined in \eqref{eq:d_e} below.
Bauwens's report \cite{Bauwens:2010} contains all components needed
for a slightly weaker version of the upper bound in Theorem~\ref{thm:quasi}
but does not state its results in terms of quasi-stochasticity.

The proof of Theorem~\ref{thm:quasi} will show that
we can improve its lower bound:
there exists $c>0$ such that, for all $\alpha\in\N_0$ and $\beta\in[0,\infty)$,
\[
  \m(\NNN^{\mathrm{q}}_{\alpha,\beta}\cap\{\omega\st K(\omega)\le\alpha+c\})
  \ge
  2^{-\alpha}/c.
\]
Therefore, an $\m$-mass of $2^{-\alpha}/c$
(the lower bound in Theorem~\ref{thm:quasi})
is concentrated already on the $\omega\in\NNN^{\mathrm{q}}_{\alpha,\beta}$
whose complexity $K(\omega)$ falls in the narrow range $[\alpha-c,\alpha+c]$.
The proof will also show that it is concentrated on a single object.

Another modification of the definition of $(\alpha,\beta)$-quasi-stochastic objects
is to replace the inequality $K(\omega\mid Q)\ge-\log Q(\{\omega\})-\beta$
by the Kolmogorov-style unconditional inequality
$K(\omega)\ge-\log Q(\{\omega\})-\beta$.
Let $\NNN^{\mathrm{qu}}_{\alpha,\beta}$
be the set of all $(\alpha,\beta)$-non-quasi-stochastic objects
under the modified definition;
we have $\NNN^{\mathrm{qu}}_{\alpha,\beta+c}\subseteq\NNN^{\mathrm{q}}_{\alpha,\beta}$ for some $c>0$.
This definition does not appear particularly natural to me,
but it leads to the following elegant version of Theorem~\ref{thm:quasi},
whose upper bound is optimal
(which implies the optimality in $\alpha$ of the upper bound
in Theorem~\ref{thm:quasi} itself).

\begin{theorem}\label{thm:unconditional}
  There exists $c>0$ such that, for all $\alpha\in\N_0$ and $\beta\in[0,\infty]$,
  \begin{equation}\label{eq:unconditional}
    2^{-\alpha-c}
    \le
    \m(\NNN^{\mathrm{qu}}_{\alpha,\beta})
    \le
    2^{-\alpha+K(\alpha)+c}.
  \end{equation}
  Moreover, if $f:\N_0\to\N_0$ is an upper semicomputable function
  such that $\sum_n2^{-f(n)}=\infty$,
  then, for infinitely many $\alpha$,
  $\m(\NNN^{\mathrm{qu}}_{\alpha,\beta})>2^{-\alpha+f(\alpha)}$
  for all $\beta\in[0,\infty]$.
\end{theorem}

As compared with Theorem~\ref{thm:quasi},
the bounds in~\eqref{eq:unconditional} do not depend on $\beta$,
and the upper bound is sharper
in that the second addend in it disappears.
The second part of Theorem~\ref{thm:unconditional},
which follows from a result by G\'acs \cite{Gacs:1983}
(see also \cite[Theorem~2]{Andreev:2017}),
shows that this upper bound cannot be improved.
Indeed, every upper semicomputable $f:\N_0\to\N_0$ with $\sum_n2^{-f(n)}<\infty$
satisfies $f\ge K-O(1)$,
and so, for any upper semicomputable $f:\N_0\to\N_0$,
\[
  \m(\NNN^{\mathrm{qu}}_{\alpha,\beta})=O\bigl(2^{-\alpha+f(\alpha)}\bigr)
  \iff
  \sum_n2^{-f(n)}<\infty.
\]
This shows that the dependence of $\m(\NNN^{\mathrm{qu}}_{\alpha,\beta})$ on $\alpha$
is irregular:
by~\eqref{eq:unconditional},
$\m(\NNN^{\mathrm{qu}}_{\alpha,\beta})\le2^{-\alpha+O(\log\log\alpha)}$
when $\alpha$ is a power of~2,
whereas $\m(\NNN^{\mathrm{qu}}_{\alpha,\beta})>\alpha2^{-\alpha}$
for infinitely many $\alpha$
(take $f(n):=\lceil\log(n+1)\rceil$).

\subsection*{General computable probability measures}

Here we will use arbitrary computable probability measures on $\XXX$,
as in \cite[Sect.~V]{Gacs/etal:2001},
in place of elementary probability measures
(used starting from \cite{Shen:1983})
as the models.

Let us fix a G\"odel numbering
(see, e.g., \cite[Chap.~4]{Shen/Vereshchagin:2003})
of all algorithms
that accept pairs in $\XXX\times\N$ as input and output rational numbers.
We say that $\pi\in\N$ is a number of a probability measure $P$ on $\XXX$
if it is a number of an algorithm that,
given $\omega\in\XXX$ and $n\in\N$,
outputs a rational number within $2^{-n}$ of $P(\{\omega\})$.
An object $\omega\in\XXX$ is \emph{$(\alpha,\beta)$-Quasi-stochastic}
(with a capital Q),
where $\alpha\in\N_0$ and $\beta\in[0,\infty)$,
if there exists a probability measure $P$ on $\XXX$
with a number $\pi$ such that
\begin{itemize}
\item
  $K(\pi)\le\alpha$;
\item
  $K(\omega\mid\pi)\ge-\log P(\{\omega\})-\beta$
\end{itemize}
(with $-\log0:=\infty$, as usual).
Let $\NNN^{\mathrm{Q}}_{\alpha,\beta}\subseteq\XXX$
be the set of all objects that are not $(\alpha,\beta)$-Quasi-stochastic.
The definition does not depend on the chosen G\"odel numbering,
apart from additive constants in $\alpha$ and $\beta$,
because any two G\"odel numberings translate into each other computably.

For some constant $c>0$, we always have
$\NNN^{\mathrm{Q}}_{\alpha+c,\beta+c}\subseteq\NNN^{\mathrm{q}}_{\alpha,\beta}$,
since a number of an elementary probability measure $Q$ can be computed from $Q$.
Therefore, the upper bounds of Theorems~\ref{thm:quasi} and~\ref{thm:unconditional}
also hold for $\NNN^{\mathrm{Q}}_{\alpha,\beta}$ and $\NNN^{\mathrm{Qu}}_{\alpha,\beta}$,
respectively (the definition of the latter is obvious);
it is not clear how to improve these upper bounds for general computable measures,
since the Epstein--Levin methods produce elementary measures anyway.
The lower bounds, however, have to be established anew.
Remember that $\mu$ is the inverse busy beaver function~\eqref{eq:mu} of Sect.~\ref{sec:main}.
We state the analogue of the lower bound in Theorem~\ref{thm:quasi};
naturally, this bound tends to 0 as $\beta\to\infty$,
unlike the lower bound of Theorem~\ref{thm:quasi}
(there exist computable probability measures with the whole of $\XXX$ as their support).

\begin{theorem}\label{thm:Q}
  There exists $c>0$ such that, for all $\alpha\in\N_0$ and $\beta\in[0,\infty)$,
  \[
    \m(\NNN^{\mathrm{Q}}_{\alpha,\beta})
    \ge
    2^{-\alpha-\mu(\beta)-c}.
  \]
\end{theorem}

\noindent
In the proof of Theorem~\ref{thm:Q} in Sect.~\ref{app:Q}
we will construct $\omega^*\in\NNN^{\mathrm{Q}}_{\alpha,\beta}$
such that $K(\omega^*)\le\alpha+\mu(\beta)+c$,
so that the lower bound in Theorem~\ref{thm:Q}
is attained on one object.
The proof will also show that Theorem~\ref{thm:Q} can be strengthened
by replacing $\NNN^{\mathrm{Q}}_{\alpha,\beta}$ with $\NNN^{\mathrm{Qu}}_{\alpha,\beta}$,
obtained by replacing the conditional complexity $K(\omega\mid\pi)$
in the second condition in the definition of $(\alpha,\beta)$-Quasi-stochasticity
with the unconditional complexity $K(\omega)$.

\section{A generalization of Theorem~\ref{thm:main}}
\label{sec:stronger}

In this section we will see stronger and simpler versions of our first key result,
Theorem~\ref{thm:main}.
But before discussing new results,
let us sort out our notation.

The most basic set related to stochasticity in this paper
is $\NNN_{\alpha,\beta}$ introduced in Sect.~\ref{sec:main}.
Its modification $\NNN^C_{\alpha,\beta}$,
obtained by replacing $K$ with $C$
in both the $\alpha$-part and the $\beta$-part
of the definition of $\NNN_{\alpha,\beta}$,
was also discussed in Sect.~\ref{sec:main}.
Kolmogorov was interested in the set $\NNN^{C\mathrm{u}}_{\alpha,\beta}$
resulting from replacing the conditional complexity $C(\omega\mid\Omega)$
in the $\beta$-part of the definition of $\NNN^C_{\alpha,\beta}$
with the unconditional complexity $C(\omega)$;
this will be further discussed in Appendix~\ref{app:Kolmogorov}.
In Sect.~\ref{sec:quasi}, we discussed modifications
$\NNN^{\mathrm{q}}_{\alpha,\beta}$ and $\NNN^{\mathrm{Q}}_{\alpha,\beta}$
obtained by replacing finite sets as models with probability measures,
and we also used $\NNN^{\mathrm{qu}}_{\alpha,\beta}$ in Theorem~\ref{thm:unconditional}
and $\NNN^{\mathrm{Qu}}_{\alpha,\beta}$ right after Theorem~\ref{thm:Q}.
In this section we will be interested in various other modifications of
$\NNN^{\mathrm{q}}_{\alpha,\beta}$,
so that our models here are elementary probability measures,
and we will need a richer and more flexible system of notation;
the new system, however, will not cover the sets $\NNN^{\mathrm{Q}}$.

We are only interested in two versions of the $\alpha$-part,
namely using $C$ or $K$.
As for the $\beta$-part,
there is a wider range of possibilities.
Let us define a \emph{deficiency} to be a lower semicomputable function $d$
mapping each pair $(\omega,Q)$,
where $\omega\in\XXX$ and $Q$ is an elementary probability measure,
to $d(\omega\mid Q)\in(-\infty,\infty]$.
Then we say that an object $\omega$ is \emph{$(\alpha,\beta)$-quasi-stochastic w.r.t.\ $d$},
where $\alpha\in\N_0$ and $\beta\in[0,\infty)$,
if there exists an elementary probability measure $Q$
such that $K(Q)\le\alpha$ and $d(\omega\mid Q)\le\beta$.
Define $\NNN_{\alpha,\beta}(d)$ to be the set of all objects
that are not $(\alpha,\beta)$-quasi-stochastic w.r.t.\ $d$.
A natural modification is to define $\NNN^C_{\alpha,\beta}(d)$
by replacing $K$ with $C$ in the definition of $\NNN_{\alpha,\beta}(d)$.

We will be interested in several specific deficiencies
and classes of deficiencies.
The basic set $\NNN^{\mathrm{q}}_{\alpha,\beta}$ is identical to
$\NNN_{\alpha,\beta}(\dL)$ with
\begin{equation}\label{eq:d_e}
  \dL(\omega\mid Q)
  :=
  -\log Q(\{\omega\})-K(\omega\mid Q)
\end{equation}
(the origin of the lower index ``e'' is the connection with e-values,
to be discussed in Sect.~\ref{sec:eight} below).
For an upper semicomputable function $F:\XXX\to\N_0$, set
\begin{equation}\label{eq:d_F}
  d_F(\omega\mid Q)
  :=
  -\log Q(\{\omega\})-F(\omega),
\end{equation}
with $-\log0:=\infty$.
In this notation, $\NNN^{\mathrm{qu}}_{\alpha,\beta}=\NNN_{\alpha,\beta}(d_K)$.

To embed our basic set $\NNN_{\alpha,\beta}$ used in Theorem~\ref{thm:main}
in the new notation,
define
\[
  d\set(\omega\mid Q)
  :=
  \begin{cases}
    d(\omega\mid Q) &
      \text{if $Q=U_{\Omega}$ for a finite set $\Omega\subseteq\XXX$}\\
    \infty & \text{otherwise}
  \end{cases}
\]
for a given deficiency $d$,
where $U_{\Omega}$ is the uniform probability measure on $\Omega$.
Then $\NNN_{\alpha,\beta}$ and $\NNN_{\alpha,\beta}(d\set_{\mathrm{e}})$
coincide to within additive constants:
for some constant $c>0$ and for all $\alpha$ and $\beta$,
\begin{equation*}
  \NNN_{\alpha+c,\beta+c}
  \subseteq
  \NNN_{\alpha,\beta}(d\set_{\mathrm{e}})
  \quad\text{and}\quad
  \NNN_{\alpha+c,\beta+c}(d\set_{\mathrm{e}})
  \subseteq
  \NNN_{\alpha,\beta}.
\end{equation*}
Kolmogorov's original set $\NNN^{C\mathrm{u}}_{\alpha,\beta}$
embeds into the new notation as $\NNN^C_{\alpha,\beta}(d\set_C)$.
Indeed, $\NNN^{C\mathrm{u}}_{\alpha,\beta}$ and $\NNN^C_{\alpha,\beta}(d\set_C)$
coincide to within additive constants:
for some constant $c>0$ and for all $\alpha$ and $\beta$,
\begin{equation*}
  \NNN^{C\mathrm{u}}_{\alpha+c,\beta+c}
  \subseteq
  \NNN^C_{\alpha,\beta}(d\set_C)
  \quad\text{and}\quad
  \NNN^C_{\alpha+c,\beta+c}(d\set_C)
  \subseteq
  \NNN^{C\mathrm{u}}_{\alpha,\beta}.
\end{equation*}
Notice that the new notation using ``$(d)$'' is more flexible
than the old notation,
since, e.g., now we can use different complexities in the $\alpha$-part and $\beta$-part,
such as using $C$ in the $\alpha$-part and $K$ in the $\beta$-part.

The following theorem says that non-quasi-stochastic objects are frequent
whenever a deficiency penalizes models at least as heavily as $d_F$,
to within multiplicative and additive constants,
for some $F$ making, unlike $K$,
the series $\sum_{\omega}2^{-F(\omega)}$ divergent.

\begin{theorem}\label{thm:penalty}
  Let $d$ be a deficiency and $F:\XXX\to\N_0$ be an upper semicomputable function
  satisfying $\sum_\omega2^{-F(\omega)}=\infty$.
  Suppose there is a constant $c>0$ such that
  \begin{equation}\label{eq:penalty}
    d_F(\omega\mid Q)
    \le
    c d(\omega\mid Q) + c
  \end{equation}
  for all $\omega$ and all elementary probability measures $Q$.
  There exists $c'>0$ such that
  $\m(\NNN_{\alpha,\beta}(d))\ge\tau(\max(\alpha,\beta))/c'$
  for all $\alpha\in\N_0$ and $\beta\in[0,\infty)$.
\end{theorem}

The non-trivial part of Theorem~\ref{thm:main} (the lower bound)
will follow immediately from Theorem~\ref{thm:penalty}
since $d_C\le2\dL\set+O(1)$.
Theorem~\ref{thm:penalty} obviously remains true for $\m(\NNN^C_{\alpha,\beta}(d))$
in place of $\m(\NNN_{\alpha,\beta}(d))$.

For the deficiencies~\eqref{eq:d_F},
Theorems~\ref{thm:unconditional} and~\ref{thm:penalty}
lead to the following dichotomy.

\begin{theorem}\label{thm:dichotomy}
  Let $F:\XXX\to\N_0$ be an upper semicomputable function.
  \begin{itemize}
  \item
    If $\sum_\omega2^{-F(\omega)}<\infty$,
    there is a constant $c>0$ such that,
    for all $\alpha\in\N_0$ and all $\beta\in[0,\infty)$,
    \begin{equation}\label{eq:dichotomy-K}
      2^{-\alpha-c}
      \le
      \m(\NNN_{\alpha,\beta}(d_F))
      \le
      2^{-\alpha+K(\alpha)+c}.
    \end{equation}
    (Only the upper bound depends on the condition on $F$.)
  \item
    If $\sum_\omega2^{-F(\omega)}=\infty$,
    there is a constant $c>0$ such that,
    for all $\alpha\in\N_0$ and all $\beta\in[0,\infty)$,
    \begin{equation}\label{eq:dichotomy-C}
      \tau(\max(\alpha,\beta))/c
      \le
      \m(\NNN_{\alpha,\beta}(d_F))
      \le
      c\tau(\alpha),
    \end{equation}
    and
    \begin{equation}\label{eq:dichotomy-mu}
      \left|
        -\log\m(\NNN_{\alpha,\beta}(d_F))
        -
        \min(\alpha,\mu(\beta))
      \right|
      \le
      c\log(\alpha+2).
    \end{equation}
  \end{itemize}
\end{theorem}

Prefix complexity $K$ and plain complexity $C$
represent different sides of the dichotomy in Theorem~\ref{thm:dichotomy}.
An example demonstrating an anomalous feature of $d_C$
(favoured by Kolmogorov in the context of finite sets as models)
as a measure of randomness was given by Levin
\cite[Sect.~4.2]{Bienvenu/etal:2022b}
in his undated letter to Kolmogorov
(written between August 1970 and January 1971):
define a probability measure on the binary strings by
\[
  Q(\{x\})
  =
  \begin{cases}
    2^{-\left|x\right|-100} & \text{if $\left|x\right|\le2^{100}$}\\
    0 & \text{otherwise};
  \end{cases}
\]
then $-\log Q(\{x\})-C(x)>100$ for all $x$.
(Formally, the binary strings should be represented by $\{0,1\}^*\setminus{\{\Box\}}$,
as on $\{0,1\}^*$ the total mass of $Q$ would be $1+2^{-100}$,
and $-\log Q(\{x\})-C(x)>100$ should be understood as $-\log Q(\{x\})-C(x)>100-c$ 
for some constant $c$, but this does not matter
as 100 can be replaced by an arbitrarily large number.)

The function $C$ can be increased considerably
without leaving the $C$-side of the dichotomy,
as demonstrated by the following lemma.

\begin{lemma}
  For a computable increasing function $h:\N_0\to\N_0$,
  \[
    \sum_\omega
    2^{-C(\omega)-h(C(\omega))}
    =
    \infty
    \iff
    \sum_n2^{-h(n)}=\infty.
  \]
\end{lemma}

\begin{proof}
  If $\sum_n2^{-h(n)}<\infty$, then $K\le h+O(1)$ on $\N_0$,
  and so $C+h(C)\ge C+K(C)-O(1)\ge K-O(1)$,
  which implies
  $
    \sum_\omega
    2^{-C(\omega)-h(C(\omega))}
    <
    \infty
  $.
  On the other hand,
  if $\sum_n2^{-h(n)}=\infty$,
  considering, for each $n$, the $2^n$ objects in $[2^{n+1}]\setminus[2^n]$
  gives, for some $c>0$,
  \[
    \sum_\omega
    2^{-C(\omega)-h(C(\omega))}
    \ge
    \sum_{n=0}^{\infty}
    2^n2^{-n-c-h(n+c)}
    =
    2^{-c}\sum_{m=c}^{\infty}2^{-h(m)}
    =\infty.
    \qedhere
  \]
\end{proof}

For example, $d_F$ leads to frequent non-quasi-stochastic objects
for $F:=C+\lceil\log(C+1)+\log\log(C+2)\rceil$
but to rare ones for $F:=C+\lceil\log(C+1)+2\log\log(C+2)\rceil$.

\section{Omnibus theorems}
\label{sec:eight}

In this section we complement the deficiencies $\dL$ and $d_F$
introduced in the previous section
by the important deficiency proposed,
in slightly different forms,
by Martin-L\"of \cite{Martin-Lof:1966}.
Let us say that a function $t$ mapping each pair $(\omega,Q)$,
where $\omega\in\XXX$ and $Q$ is an elementary probability measure,
to $\N_0\cup\{\infty\}$
is a \emph{randomness p-test}
if it is lower semicomputable
and satisfies
\[
  Q(\{\omega\st t(\omega,Q)\ge n\})\le2^{-n}
\]
for all $Q$ and all $n\in\N$.
Let $\dML$ be a largest, to within an additive constant,
randomness p-test.
We fix $\dML$ and refer to $\dML(\omega\mid Q)$
as the \emph{randomness p-deficiency} of $\omega$ under $Q$.
For the uniform probability measures $U_{\Omega}$ on finite sets $\Omega$,
\begin{equation}\label{eq:ML}
  \dML(\omega\mid U_{\Omega})=\log\left|\Omega\right|-C(\omega\mid\Omega)+O(1)
\end{equation}
for $\omega\in\Omega$
(this is Martin-L\"of's expression of the conditional version of Kolmogorov's deficiency
in terms of tests
\cite[Sect.~II, Theorem on p.~607]{Martin-Lof:1966};
see also \cite[Appendix~1, Proposition~11]{Vovk:2025ML}).

The randomness p-deficiency $\dML(\cdot\mid Q)$
is a p-variable for $Q$ on the logarithmic scale,
meaning that $2^{-\dML(\cdot\mid Q)}$ is a p-variable for $Q$,
as defined in, e.g., \cite[Definition 1.2(ii)]{Ramdas/Wang:2025}.
Similarly, the function $\dL(\cdot\mid Q)$ defined by \eqref{eq:d_e}
is an e-variable for $Q$ on the logarithmic scale,
meaning that $2^{\dL(\cdot\mid Q)}$ is an e-variable for $Q$,
as defined in, e.g., \cite[Definition 1.2(i)]{Ramdas/Wang:2025}.
Both $\dML$ and $\dL$ are the largest, to within an additive constant,
lower semicomputable functions in these classes
(of the p- and e-variables, respectively, on the logarithmic scale).
Therefore, it is natural to call $\dL(\omega\mid Q)$
the \emph{randomness e-deficiency} of $\omega$ under $Q$.
The functions $\dML$ and $\dL$ can be interpreted
as universal versions of p-values and e-values, respectively.

Now we can state our first omnibus theorem covering eight cases;
it will be proved in Sect.~\ref{app:eight}.

\begin{theorem}\label{thm:eight}
  Let $F:\XXX\to\N_0$ be an upper semicomputable function
  such that $\sum_\omega2^{-F(\omega)}=\infty$
  (e.g., $F=C$).
  There exists $c>0$ such that, for all $\alpha\in\N_0$ and $\beta\in[0,\infty)$,
  \begin{align}
    2^{-\alpha}/c
    &\le
    \m(\NNN_{\alpha,\beta}(\dL))
    \le
    c\left(2^{-\alpha+K(\alpha)}+2^{-2^{\beta/2-c}}\right),
    \label{eq:K-e}\\
    2^{-\alpha}/c
    &\le
    \m(\NNN_{\alpha,\beta}(\dML))
    \le
    c\left(2^{-\alpha+K(\alpha)}+2^{-2^{\beta/2-c}}\right),
    \label{eq:K-p}\\
    2^{-\alpha}/c
    &\le
    \m(\NNN_{\alpha,\beta}(d_K))
    \le
    c2^{-\alpha+K(\alpha)},
    \label{eq:K-K}\\
    \left(2^{-\alpha}+\tau(\max(\alpha,\beta))\right)/c
    &\le
    \m(\NNN_{\alpha,\beta}(d_F))
    \le
    c\left(2^{-\alpha+K(\alpha)}+\tau(\max(\alpha,\beta))\right),
    \label{eq:K-C}
  \end{align}
  and
  \begin{align}
    2^{-\alpha-K(\alpha)}/c
    &\le
    \m(\NNN^C_{\alpha,\beta}(\dL))
    \le
    c\left(2^{-\alpha}+2^{-2^{\beta/2-c}}\right),
    \label{eq:C-e}\\
    2^{-\alpha-K(\alpha)}/c
    &\le
    \m(\NNN^C_{\alpha,\beta}(\dML))
    \le
    c\left(2^{-\alpha}+2^{-2^{\beta/2-c}}\right),
    \label{eq:C-p}\\
    2^{-\alpha-K(\alpha)}/c
    &\le
    \m(\NNN^C_{\alpha,\beta}(d_K))
    \le
    c2^{-\alpha},
    \label{eq:C-K}\\
    \left(2^{-\alpha-K(\alpha)}+\tau(\max(\alpha,\beta))\right)/c
    &\le
    \m(\NNN^C_{\alpha,\beta}(d_F))
    \le
    c\left(2^{-\alpha}+\tau(\max(\alpha,\beta))\right).
    \label{eq:C-C}
  \end{align}
\end{theorem}

The eight double inequalities in Theorem~\ref{thm:eight} have an interesting structure.
The bounds for $\dL$ and $\dML$ are identical,
for both conditions $K(Q)\le\alpha$ and $C(Q)\le\alpha$,
despite e-values and p-values being regarded as very different in statistics.
Replacing the condition $K(Q)\le\alpha$ with $C(Q)\le\alpha$
moves the factor $2^{K(\alpha)}$ from the upper bound to the lower bound.
For $\dL$, $\dML$, and $d_K$,
non-quasi-stochastic objects are rare:
their universal measure is $2^{-\alpha+O(\log(\alpha+2))}$
(for $\dL$ and $\dML$, as soon as $\beta\ge2\log(\alpha+1)+c'$).
For $d_F$, they are frequent,
in agreement with Theorem~\ref{thm:dichotomy}.
The first sandwich, \eqref{eq:K-e}, is Theorem~\ref{thm:quasi},
and~\eqref{eq:K-K} is the first part of Theorem~\ref{thm:unconditional}.

The terms $2^{-\alpha}$ and $2^{-\alpha\pm K(\alpha)}$ in~\eqref{eq:K-C} and~\eqref{eq:C-C}
matter only for huge $\beta$.
For a fixed computable function $g$ and all $\beta\le g(\alpha)$,
they are dominated by the terms $\tau(\max(\alpha,\beta))$,
and both measures in~\eqref{eq:K-C} and~\eqref{eq:C-C} are $\Theta(\tau(\alpha))$
(by Lemma~\ref{lem:g}), as in Theorem~\ref{thm:main}.
However, $\tau(\max(\alpha,\beta))\to0$ as $\beta\to\infty$,
while the measures in~\eqref{eq:K-C} and~\eqref{eq:C-C} never fall below
$2^{-\alpha}/c$ and $2^{-\alpha-K(\alpha)}/c$, respectively.
In~\eqref{eq:dichotomy-mu},
these terms correspond to the $\alpha$ in $\min(\alpha,\mu(\beta))$.

The upper bounds in~\eqref{eq:K-C} and~\eqref{eq:C-C}
hold for any upper semicomputable $F:\XXX\to\N_0$
(only $F\ge0$ is used);
the condition $\sum_\omega2^{-F(\omega)}=\infty$ is only needed
for the terms $\tau(\max(\alpha,\beta))$ in the lower bounds.
The sandwich \eqref{eq:K-C} implies both~\eqref{eq:dichotomy-C} and~\eqref{eq:dichotomy-mu},
but we will give simple independent derivations
for the latter two in Sect.~\ref{app:stronger}.

Theorem~\ref{thm:eight} is about non-quasi-stochastic objects.
The situation for nonstochastic objects is much simpler,
as the following result (proved in Sect.~\ref{app:eight-set}) shows.

\begin{theorem}\label{thm:eight-set}
  Let $F$ be as in Theorem~\ref{thm:eight}.
  There exists $c>0$ such that,
  for all $\alpha\in\N_0$, all $\beta\in[0,\infty)$,
  and each $d\in\{\dL,\dML,d_K,d_F\}$,
  \begin{align}
    \m(\NNN_{\alpha,\beta}(d\set))
    &\ge
    \left(2^{-\alpha}+\tau(\max(\alpha,\beta))\right)/c,
    \label{eq:K-set}\\
    \m(\NNN^C_{\alpha,\beta}(d\set))
    &\ge
    \left(2^{-\alpha-K(\alpha)}+\tau(\max(\alpha,\beta))\right)/c,
    \label{eq:C-set}
  \end{align}
  and both measures are at most
  \begin{equation}\label{eq:set-upper}
    c\left(2^{-\alpha+c\log(\alpha+2)}+\tau(\max(\alpha,\beta))\right).
  \end{equation}
\end{theorem}

In all eight cases covered by Theorem~\ref{thm:eight-set},
$(\alpha,\beta)$-nonstochastic objects are frequent under the universal measure
provided $\beta\le g(\alpha)$ for a computable $g$:
their universal measure is $\Theta(\tau(\alpha))$.
Moreover, the theorem
(namely, \eqref{eq:K-set} and \eqref{eq:set-upper} with $d:=\dL$)
shows that our first key result \eqref{eq:super} can be strengthened to
\begin{equation}\label{eq:alternative}
  \m(\NNN_{\alpha,\beta})
  =
  \Theta
  \left(
    \tau(\max(\alpha,\beta))
    +
    2^{-\alpha+O(\log(\alpha+2))}
  \right),
\end{equation}
the second addend on the right-hand side kicking in for extremely large $\beta$.
This expression for $\m(\NNN_{\alpha,\beta})$ complements Theorem~\ref{thm:beta},
which was given in terms of $\mu$ rather than $\tau$.

\subsection*{Statements in terms of the conditional universal measure}

In Theorem~\ref{thm:main}
and further results of this kind,
such as \eqref{eq:K-C}, \eqref{eq:C-C}, and \eqref{eq:K-set}--\eqref{eq:C-set},
we interpreted a set such as $\NNN_{\alpha,\beta}$
having an $\m$-share of $\tau(\alpha)$ as being frequent.
However, $\m(\NNN_{\alpha,\beta})$ still tends to 0 as $\alpha\to\infty$,
albeit slowly.
In this subsection these results will be restated in terms
of a constant proportion of nonstochastic or non-quasi-stochastic objects.

\begin{theorem}\label{thm:eight-high}
  Let $F$ be as in Theorem~\ref{thm:eight}.
  There exists $c>0$ such that,
  for all $\alpha,k\in\N_0$ and $\beta\in[0,\infty)$,
  each of the bounds in Theorems~\ref{thm:eight} and~\ref{thm:eight-set},
  of the form $\m(\NNN)\ge L/c$ or $\m(\NNN)\le c U$,
  remains true when $\m(\NNN)$ is replaced by $\m(\NNN\cap B)$,
  where $B$ is either $\{\omega\st K(\omega)>k\}$ or $\{\omega\st C(\omega)>k\}$,
  and $L$ and $U$ are replaced by $\min(L,\tau(k))$ and $\min(U,\tau(k))$, respectively.
\end{theorem}

See Sect.~\ref{app:eight-high} for the proof.
In view of Lemma~\ref{lem:g} and \eqref{eq:tau-sigma},
Theorem~\ref{thm:eight-high} immediately implies the following corollary.

\begin{corollary}
  Let $F$ be as in Theorem~\ref{thm:eight}.
  There exists $c>0$ such that,
  for all $k\in\N_0 $ and all $\alpha\in\N_0,\beta\in[0,\infty)$
  satisfying $\max(\alpha,\beta)\le2^{2^{k}}$,
  we have $\m(A\mid B)>c$ whenever
  \begin{itemize}
  \item
    $A$ is one of the sets $\NNN_{\alpha,\beta}(d_F)$, $\NNN^C_{\alpha,\beta}(d_F)$,
    $\NNN_{\alpha,\beta}(d\set)$, or $\NNN^C_{\alpha,\beta}(d\set)$,
    with $d\in\{\dL,\dML,d_K,d_F\}$, and
  \item
    $B$ is either $\{\omega\st K(\omega)>k\}$ or $\{\omega\st C(\omega)>k\}$.
  \end{itemize}
\end{corollary}

As before, the tower $2^{2^k}$ can be replaced by $g(k)$ for any computable increasing $g$;
for a higher tower, \eqref{eq:frequent} becomes essentially a special case of this corollary.
The corollary shows that the nonstochastic objects,
as well as non-quasi-stochastic objects defined in terms of $d_F$,
make up at least a constant $\m$-fraction of all objects of high complexity.
The statement about $d_F$ is the main one for the proofs in this paper,
since we reduce the cases of $d\set$ to $d_F$ via Theorem~\ref{thm:penalty}.

\section{Conclusion}

Let me summarize the main features of this paper's structure-free picture.
The universal measure of $(\alpha,\beta)$-nonstochastic objects
hardly depends on $\beta$
(unless $\beta$ is huge, in a busy-beaver sense, as compared with $\alpha$)
and tends to 0 extremely slowly as $\alpha\to\infty$,
more slowly than any decreasing computable function that tends to 0.
Above any complexity level $k\ge\alpha$
a significant part of $\m$ is allocated to $(\alpha,\beta)$-nonstochastic objects:
\[
  \m\bigl(\NNN_{\alpha,\beta}\cap\{\omega\st K(\omega)>k\}\bigr)
  =
  \Theta\bigl(\m(\{\omega\st K(\omega)>k\})\bigr)
\]
(provided $\beta$ is computably bounded in $\alpha$).
In sharp contrast with $(\alpha,\beta)$-nonstochastic objects,
the universal measure of $(\alpha,\beta)$-non-quasi-stochastic objects
is exponentially small in $\alpha$ provided $\beta\ge2\log\alpha$,
both in the case where the quality of models is measured
by universal e-values
and in the case where it is measured by universal p-values.

\subsection*{Acknowledgments}

This paper was conceived during my work on the V'yugin Special Issue
\cite{Vyugin:2026}.
Vladimir V'yugin, one of my teachers in the algorithmic theory of randomness,
was a key contributor to the literature on Kolmogorov's $(\alpha,\beta)$-stochasticity.

Many thanks to Sasha Shen for several discussions
of Kolmogorov's definition of stochastic objects,
and also for his record of Kolmogorov's talk 45 years ago.

I acknowledge the use of Claude Opus~5--5.5 and Fable~5.1
in exploring proof ideas and developing proofs, which I reviewed carefully,
and in checking the paper.
I take full responsibility for this paper's claims and statements,
including mathematical statements and their proofs.
More generally, I have verified the correctness and originality of all content,
including references.

\subsection*{History of arXiv revisions}

\begin{description}
\item[v1]
  This is a basic 8-page version motivated by Kolmogorov's 1981 questions,
  which are stated in the appendix (without answering them).
  Unusually in light of the literature on stochasticity,
  Kolmogorov's setting is structure-free.
  The main result (Theorem~1) is that $\m(\NNN_{\alpha,\beta})=\Theta(\tau(\alpha))$
  provided $\beta\le2^{2^{\alpha}}$
  (which can be replaced by an analogous tower of any height).
  The key technique used in the proof
  is based on G\'acs et al.'s observation that the complexity of stochastic objects is simple
  (Lemma~5 bounding $K(K(\omega))$ in terms of $\alpha$ and $\beta$).
\item[v2]
  Kolmogorov's questions,
  complemented by answers to the specific questions,
  are moved to the main paper (Sect.~2).
  The main result stays the same (Theorem~4 in Sect.~3),
  but now it is complemented by two more results (Sect.~4),
  both motivated by Kolmogorov's questions;
  the first result estimates
  $\m(\NNN_{\alpha,\beta}\cap\{\omega\mid K(\omega)>k\})$
  (Theorem~7)
  and the second estimates
  $\m(\NNN_{\alpha,\beta}\cap\{\omega\mid K(\omega)\le k\})$
  (Theorem~8).
  Another topic added in this version is the structure-free setting
  for quasi-stochasticity,
  which has already been explored by Epstein and Levin.
  In Sect.~5 I state a simple corollary of their result.

  Lemma~13 of v2 implements Shen's technique of the ``first uncovered object'',
  but Lemma~\ref{lem:BB} below does it better;
  using the latter instead of the former improves (10) in v2 to
  \begin{equation}\label{eq:low-lower-bis}
    \m(\LLL_{\alpha,\beta,k})
    \ge
    2^{-\alpha-c}
    \quad\text{if}\quad
    k\ge\alpha+c,
  \end{equation} where $\LLL_{\alpha,\beta,k}:=\NNN_{\alpha,\beta}\cap\{\omega\st K(\omega)\le k\}$.
\item[v3]
  This is a streamlined version, with some results removed.
  In particular, to make the paper more focused, Theorem~8 of v2 has been removed.
  It is correct, but can be strengthened by replacing its equation~(10)
  with \eqref{eq:low-lower-bis}
  (i.e., removing unnecessary $K(\alpha)$ terms).
  Kolmogorov's questions are moved to Appendix~\ref{app:Kolmogorov}.
  Several new results are added, including omnibus theorems.
  The main technique for proving that nonstochastic objects are frequent
  is now different, as discussed at the beginning of Appendix~\ref{app:proofs}.
\end{description}

\appendix
\section{Proofs}
\label{app:proofs}

In this appendix we will prove
all the theorems stated in the main part of this paper.
The basic space $\XXX$ is interpreted as $\N$.

In the proofs, $c_1,c_2,\dots$ will stand for various positive constants;
numbering them facilitates checking the absence of vicious circles
in their definitions.
Each constant is absolute or depends only on previously defined constants,
except for a few explicitly chosen ``large enough'' later in the same proof,
after constants that do not depend on them.

Let me first describe the main techniques used in the proofs in this appendix
and in the proof of Theorem~\ref{thm:Kolmogorov} in Appendix~\ref{app:Kolmogorov}.
Proofs of the upper bounds often rely on the standard techniques of halting cuts,
either cuts of Chaitin's number $\Chaitin$ as used, e.g.,
by Epstein and Levin \cite[Lemma~10]{Epstein:2021},
or cuts of ``bounded complexity lists''
(see, e.g., \cite[Sect.~III.E]{Gacs/etal:2001} and \cite[Sect.~4]{Vereshchagin/Shen:2017}).
These standard techniques are used in the proofs of the upper bounds
of Theorems~\ref{thm:quasi}--\ref{thm:unconditional},
\ref{thm:dichotomy}, \ref{thm:eight}--\ref{thm:eight-high},
and~\ref{thm:Kolmogorov} (sometimes indirectly).

The main new technique introduced in this paper
is used in proofs of lower bounds of numerous theorems
(including, directly or indirectly,
Theorems~\ref{thm:main}, \ref{thm:beta}--\ref{thm:high},
\ref{thm:penalty}--\ref{thm:dichotomy},
and~\ref{thm:eight}--\ref{thm:eight-high}).
It develops G\'acs et al.'s observation
that the complexity of stochastic objects is simple
\cite[Lemmas~III.8 and~III.10]{Gacs/etal:2001}:
$K(K(\omega))$ cannot be very large as compared with $\alpha$ and $\beta$
for $(\alpha,\beta)$-stochastic objects.
Combined with the fact that the universal measure often generates $\omega$
with a large $K(K(\omega))$,
this implies that nonstochastic objects are abundant.
Early versions of this paper (such as v1 and v2 of the arXiv report)
relied literally on this observation.
This mutated into using the observation that, for stochastic objects $\omega$,
$K(\omega)\approx C(\omega)$,
while the universal measure often generates $\omega$
with very different $K(\omega)$ and $C(\omega)$.
Because of Solovay's result
\[
  K(\omega)-C(\omega)
  =
  K(K(\omega)) + O(K(K(K(\omega))))
\]
(see, e.g., \cite[Sect.~4.7.4]{Shen/etal:2017book}),
there is very little difference between the two observations.
It turns out that in the second observation
$C$ can be replaced by any upper semicomputable function $F$
such that $\sum_{\omega}2^{-F(\omega)}=\infty$
(see~\eqref{eq:chain} in the proof of Theorem~\ref{thm:penalty}
and Lemma~\ref{lem:Kraft} below).
Now the properties $K(K(\omega))\approx0$ and $K(\omega)\approx F(\omega)$
of an object $\omega$ become genuinely different.

Another new technique used for several lower bounds
involves a simple modification of Shen's technique
of the ``first uncovered object''
\cite[Theorem~2(b)]{Shen:1983}.
Shen's technique was also used in early versions of this paper
(including v1 and v2 of the arXiv report),
but its simplification has strengthened several results.
It is used in proving the lower bounds
in Theorems~\ref{thm:beta}, \ref{thm:quasi}--\ref{thm:Q},
\ref{thm:dichotomy}, \ref{thm:eight}--\ref{thm:eight-high},
and~\ref{thm:Kolmogorov}.

\subsection{The tail of the universal measure}
\label{app:tail}

This section proves key results about the function $\tau$,
which plays a prominent role in several bounds,
and related functions.
Some of the proofs employ the great stability of $\tau$;
it hardly changes over huge regions of its domain.
We start with the trivial observation
that the composition of a computable Markov kernel with the universal measure
is dominated by the universal measure.

\begin{lemma}\label{lem:kernel}
  Let $(\kappa_n)_{n\in\N}$ be probability measures on $\N$
  such that the mapping $(n,i)\in\N^2\mapsto\kappa_n(\{i\})$ is computable.
  There is $c_{\kappa}>0$ such that, for all $A\subseteq\N$,
  \[
    \m(A)
    \ge
    c_{\kappa}\sum_{n\in\N}\m(n)\kappa_n(A).
  \]
\end{lemma}

\begin{proof}
  The function $i\mapsto\sum_n \m(n)\kappa_n(\{i\})$ is a lower semicomputable semimeasure,
  and the universality of $\m$ gives $\m(i)\ge c_{\kappa}\sum_n \m(n)\kappa_n(i)$ for all $i$.
  It remains to sum over $i\in A$.
\end{proof}

Our first application of Lemma~\ref{lem:kernel}
is to demonstrate the great stability of $\tau$
asserted in Lemma~\ref{lem:g}.

\begin{proof}[Proof of Lemma~\ref{lem:g}]
  Applying Lemma~\ref{lem:kernel} to $\kappa_n$ set to $\delta_{g(n)+1}$,
  the point mass at $g(n)+1$,
  we have
  \begin{multline*}
    \tau(g(t))
    =
    \m(\{i\st i>g(t)\})
    \ge
    c_{\kappa}
    \sum_{n\in\N}
    \m(n)
    1_{\{g(n)\ge g(t)\}}\\
    \ge
    c_{\kappa}
    \sum_{n\in\N}
    \m(n)
    1_{\{n\ge t\}}
    =
    c_{\kappa}
    \sum_{n\ge t}\m(n)
    \ge
    c_{\kappa}
    \tau(t).
    \qedhere
  \end{multline*}
\end{proof}

Let us check the relation between $\tau$ and $\sigma$
stated in Sect.~\ref{sec:main} directly
(the lower bound on $\sigma$ will also follow from the next lemma,
but the direct proof is easier).

\begin{proof}[Proof of Lemma~\ref{lem:sigma-tau}]
  The lower bound in \eqref{eq:tau-sigma}
  follows from Lemma~\ref{lem:kernel}
  applied to $A:=\{n\st K(n)>t\}$
  and the uniform probability measures $\kappa_k$ on $[2^k]$.
  At most $2^{t}$ numbers $n$ have $K(n)\le t$,
  so for $k\ge t+1$ we have
  $\kappa_k(A)\ge1/2$;
  therefore,
  \[
    \sigma(t)
    =
    \m(\{n\st K(n)>t\})
    =
    \m(A)
    \ge
    \cnew
    \sum_{k=t+1}^{\infty}
    \frac{\m(k)}{2}
    =
    \frac{\cold}{2}
    \tau(t).
  \]
  The upper bound follows from
  \begin{equation*}
    \sigma(t)
    =
    \m(\{n\st K(n)>t\})
    \le
    \m(\{n\st n>t-\cnew\})
    =
    \tau(t-\cold)
    \le
    \cnew\tau(t),
  \end{equation*}
  where we have used Lemma~\ref{lem:g};
  this chain assumes $t\ge\colder$
  (smaller $t$ are covered by $\sigma\le1$ and $\tau(t)>0$).
\end{proof}

The following lemma (which will be one of our key technical tools)
generalizes the lower bound in~\eqref{eq:tau-sigma}.

\begin{lemma}\label{lem:Kraft}
  Let $F:\XXX\to\N_0$ be an upper semicomputable function
  such that $\sum_\omega2^{-F(\omega)}=\infty$.
  There exists a constant $\cnew>0$ such that,
  for all $t\in\N_0$,
  \begin{equation}\label{eq:Kraft}
    \m(\{\omega\st K(\omega)-F(\omega)>t\})
    \ge
    \cold\tau(t).
  \end{equation}
\end{lemma}

\noindent
For $F:=0$, this is the lower bound in~\eqref{eq:tau-sigma}.
On the other hand,
if $\sum_\omega2^{-F(\omega)}<\infty$, $K-F$ is bounded above
and so the left-hand side of~\eqref{eq:Kraft} vanishes from some $t$ on.

\begin{proof}
  Let $F_s$, $s\in\N_0$, be a computable approximation to $F$ from above
  (so that $F_s(\omega)\in\N_0\cup\{\infty\}$ decreases to $F(\omega)$ as $s\to\infty$).
  For $k\in\N$, find the first pair $(s,n)$
  (in a fixed computable enumeration of $\N_0\times\N$)
  such that $Z_k:=\sum_{\omega\in[n]}2^{-F_s(\omega)}\ge2^k$
  (such pairs exist since $\sum_\omega2^{-F(\omega)}=\infty$),
  and let $\kappa_k$ be the probability measure on $[n]$
  defined by $\kappa_k(\{\omega\}):=2^{-F_s(\omega)}/Z_k$.
  Since $\kappa_k(\{\omega\})\le2^{-k-F(\omega)}$ for all $\omega$,
  we have
  \[
    \kappa_k\bigl(\{\omega\st K(\omega)-F(\omega)\le t\}\bigr)
    \le
    \sum_{\omega:F(\omega)\ge K(\omega)-t}
    2^{-k-F(\omega)}
    \le
    \sum_{\omega}2^{-k-K(\omega)+t}
    \le
    2^{t-k}.
  \]
  Therefore, $\kappa_k(A)\ge1/2$ for $A:=\{\omega\st K(\omega)-F(\omega)>t\}$
  and all $k\ge t+1$,
  and Lemma~\ref{lem:kernel} gives
  \[
    \m(A)
    \ge
    \frac{c_{\kappa}}{2}\sum_{k=t+1}^{\infty}\m(k)
    =
    \frac{c_{\kappa}}{2}\tau(t).
    \qedhere
  \]
\end{proof}

The next lemma spells out the easy part of \eqref{eq:Andreev}.

\begin{lemma}\label{lem:Andreev}
  For some constant $\cnew>0$,
  $\tau \ge\cold2^{-\mu}$.
\end{lemma}

\begin{proof}
  The inequality follows from
  \[
    \tau(t)
    =
    \m(\{n\st n>t\})
    \ge
    \m(n^*)
    \ge
    \cold 2^{-K(n^*)}
    =
    \cold 2^{-\mu(t)},
  \]
  where $n^*>t$ is any number that achieves $K(n^*)=\mu(t)$.
\end{proof}

\subsection{Proof of Theorem~\ref{thm:main}}
\label{app:main}

We will prove Theorem~\ref{thm:main} in the form~\eqref{eq:super}.
This implies the other version of Theorem~\ref{thm:main}
given right after its statement.
This was already mentioned in the main paper,
but let me spell out all details.
Let $g$ be a computable function;
we assume, without loss of generality,
that $g$ is increasing and $g(n)\ge n$
for all $n$ in its domain.
Lemma~\ref{lem:g} gives, for $\beta\le g(\alpha)$,
\[
  \tau(\max(\alpha,\beta))
  \ge
  \tau(g(\alpha))
  \ge
  c_g\tau(\alpha).
\]
It remains to combine this with~\eqref{eq:super}.
The main statement of the theorem uses $g(\alpha):=2^{2^{\alpha}}$.

Let us start with the easy upper bound in \eqref{eq:super}.
For some constant $\cnew$\clabel{c:singleton},
\begin{equation}\label{eq:singleton}
  K(\{\omega\})\le K(\omega)+\cref{c:singleton}.
\end{equation}
Every $\omega\in\NNN_{\alpha,\beta}$ satisfies $K(\omega)>\alpha-\cref{c:singleton}$
(otherwise $\Omega:=\{\omega\}$ would have witnessed its $(\alpha,0)$\-/stochasticity).
Therefore, every $\omega\in\NNN_{\alpha,0}$ satisfies $2\log\omega\ge\alpha-\cnew$,
which implies
\begin{equation*}
  \omega \ge 2^{\alpha/2-\cold/2} \ge \alpha+1
\end{equation*}
for all $\alpha\ge\cnew$.
This implies
\[
  \m(\NNN_{\alpha,\beta})
  \le
  \m(\NNN_{\alpha,0})
  \le
  \m(\{n\st n>\alpha\})
  =
  \tau(\alpha)
\]
from $\alpha=\cold$ on.
(This is sufficient since the remaining finitely many $\alpha$s
can be absorbed into the constant in \eqref{eq:super}.)

It remains to prove the lower bound.
It is an easy corollary of Theorem~\ref{thm:penalty},
which will be proved in Sect.~\ref{app:stronger}
independently of this proof.
For some constants $\cnew$ and $\cnew$,
for all finite sets $\Omega\subseteq\XXX$ and all $\omega\in\Omega$,
\begin{equation}\label{eq:to-simplify}
  \log\left|\Omega\right|-C(\omega)
  \le
  \log\left|\Omega\right|-C(\omega\mid\Omega)+\colder
  \le
  2\Bigl(\log\left|\Omega\right|-K(\omega\mid\Omega)\Bigr)+\cold
\end{equation}
(\cite[Proposition~17]{Vovk:2025ML};
the proof in \cite[Proposition~1]{Novikov:arXiv1608} is applicable);
therefore,
\[
  d_C(\omega\mid Q)
  \le
  2d\set_{\mathrm{e}}(\omega\mid Q)+\cnew,
\]
and Theorem~\ref{thm:penalty} gives the lower bound.

The version of the theorem
with $\NNN^C_{\alpha,\beta}$ in place of $\NNN_{\alpha,\beta}$
is proved in the same way
(and the second inequality in \eqref{eq:to-simplify} is not needed).

\subsection{Proof of Theorem~\ref{thm:high}}
\label{app:high}

Set
$
  A
  :=
  \NNN_{\alpha,\beta}\cap\{\omega\st K(\omega)>k\}
$.
Let us check that, analogously to \eqref{eq:super},
\begin{equation}\label{eq:bounds}
  \tau(\max(\alpha,\beta,k))/\cnew\clabel{c:super-high}
  \le
  \m(A)
  \le
  \cold\tau(\max(\alpha,k)).
\end{equation}
Under the condition $\beta\le g(\alpha)$
for a computable increasing $g$,
Lemma~\ref{lem:g} then gives
\[
  \m(A)
  \ge
  \tau(\max(\alpha,\beta,k))/\cold
  \ge
  \tau(\max(g(\alpha),k))/\cold
  \ge
  \cnew\tau(\max(\alpha,k))
\]
for some $\cold>0$
(again assuming, without loss of generality,
$g(n)\ge n$ for all $n$).

We start with the upper bound in \eqref{eq:bounds}.
Since $\tau$ is decreasing,
$\tau(\max(\alpha,k))=\min(\tau(\alpha),\tau(k))$,
so it suffices to bound $\m(A)$ above by a multiple of each of
$\tau(\alpha)$ and $\tau(k)$.
From $A\subseteq\NNN_{\alpha,\beta}$ and the upper bound in \eqref{eq:main}
we get $\m(A)\le c\tau(\alpha)$;
this part of Theorem~\ref{thm:main} needs no condition on $\beta$,
as $\NNN_{\alpha,\beta}\subseteq\NNN_{\alpha,0}$.
And since $A\subseteq\{\omega\st K(\omega)>k\}$,
\eqref{eq:tau-sigma} gives
$
  \m(A)
  \le
  \sigma(k)
  \le
  \cnew\tau(k)
$.

It remains to establish the lower bound in \eqref{eq:bounds}.
Set $\alpha':=\max(\alpha,k+\cref{c:singleton})$
(see \eqref{eq:singleton} for $\cref{c:singleton}$).
Since $\alpha'\ge\alpha$, we have $\NNN_{\alpha',\beta}\subseteq\NNN_{\alpha,\beta}$;
and since every $\omega\in\NNN_{\alpha',\beta}$ satisfies
$K(\omega)>\alpha'-\cref{c:singleton}\ge k$,
we have $\NNN_{\alpha',\beta}\subseteq A$.
Therefore, the lower bound in~\eqref{eq:super} gives
\[
  \m(A)
  \ge
  \m(\NNN_{\alpha',\beta})
  \ge
  \tau(\max(\alpha',\beta))/\cnew
  \ge
  \tau\bigl(\max(\alpha,\beta,k)+\cref{c:singleton}\bigr)/\cold,
\]
and Lemma~\ref{lem:g} gives
$\m(A)\ge\tau(\max(\alpha,\beta,k))/\cref{c:super-high}$
for a large enough $\cref{c:super-high}$
(which should also cover the upper bound in \eqref{eq:bounds}).

\subsection{Proof of Theorems~\ref{thm:quasi} and~\ref{thm:unconditional}}
\label{app:quasi}

Before we start the proof,
let me explain in greater detail the reason for the stark contrast
between Theorem~\ref{thm:main}
(there are many nonstochastic objects)
and Theorem~\ref{thm:quasi}
(there are few non-quasi-stochastic objects).
A finite set $\Omega$ used as a model
determines $\log\left|\Omega\right|$,
and so it essentially determines the complexity $K(\omega)$
of its typical elements $\omega\in\Omega$.
As $\Omega$ is simple,
the complexity of stochastic objects is also simple;
the objects whose complexity is complex are nonstochastic.
A probability measure $Q$, on the other hand,
does not have to determine $-\log Q(\{\omega\})$:
it can spread its mass over objects of different complexities.

\begin{proof}[Proof of the upper bound of Theorem~\ref{thm:quasi}]
  In this proof we employ the technique of halting cuts.
  Let $U$ be the optimal prefix-free decoder,
  so that $K(\omega)=\min\{\left|p\right|\st U(p)=\omega\}$ for all $\omega$.
  Let $\Chaitin$ (Chaitin's number $\Omega$)
  be its halting probability, i.e.,
  the sum of the \emph{weights} $2^{-\left|p\right|}$ over all $p$ such that $U(p)$ halts.
  Fix a process of computing $U(p)$ for all $p\in\{0,1\}^*$,
  and let $\Chaitin_t$, $t\in\N_0$, be its halting probability in $t$ steps,
  i.e., the sum of $2^{-\left|p\right|}$ over all $p$
  such that $U(p)$ halts in at most $t$ steps.
  For each $k\in\N_0$, define
  \[
    T_k
    :=
    \min
    \left\{
      t\st\Chaitin_t\ge\lfloor2^k\Chaitin\rfloor2^{-k}
    \right\};
  \]
  in words, $T_k$ is the time of computing the $k$-bit truncation of $\Chaitin$.
  Next $k(p)$, $p\in\{0,1\}^*$, is the \emph{slice} of $p$,
  defined to be the $k$ such that the number of steps in which $U(p)$ halts
  belongs to the interval $(T_k,T_{k+1}]$.
  (This is well defined since $T_0=0$ and $T_k\to\infty$ as $k\to\infty$.)

  Next we define our models.
  For $k\in\N_0$, define a measure $Q_k$ on $\XXX$
  by setting $Q_k(\{\omega\})$ to $2^k$ times
  the total weight of $p$ in slice $k$ satisfying $U(p)=\omega$.
  This definition only ensures $\sum_{\omega}Q_k(\{\omega\})\le1$,
  so we increase the value $Q_k(\{\omega_0\})$ for a fixed $\omega_0\in\XXX$
  to make $Q_k$ a probability measure.
  The values of $Q_k$ are dyadic rationals, so $Q_k$ are elementary probability measures.

  To bound $\m(\NNN^{\mathrm{q}}_{\alpha,\beta})$ from above,
  let us consider $\omega\in\NNN^{\mathrm{q}}_{\alpha,\beta}$,
  set $k:=k(\omega):=k(p)$ for some shortest $U$-description $p$ of $\omega$
  (chosen arbitrarily but fixed for each $\omega$),
  and suppose $k\le\alpha-K(\alpha)-\cnew\clabel{c:C}$
  for a large $\cref{c:C}$.
  We will use $Q_k$ as our model for $\omega$;
  its complexity can be bounded in two ways:
  \begin{itemize}
  \item
    Since a self-delimiting description of $Q_k$
    (for some prefix-free decoder, not necessarily $U$)
    is provided by a shortest $U$\-/description of $k$
    and the first $k+1$ bits of (the binary expansion of) $\Chaitin$,
    we have
    \begin{equation}\label{eq:first}
      K(Q_k)\le k+K(k)+\cnew.
    \end{equation}
  \item
    Another self-delimiting description of $Q_k$
    consists of a shortest $U$\-/description of $\alpha$
    (which gives us both $\alpha$ and $K(\alpha)$),
    a shortest $U$\-/description of $\alpha-K(\alpha)-k$,
    and the first $k+1$ bits of $\Chaitin$;
    this gives
    \begin{equation}\label{eq:alpha-bound}
      K(Q_k)
      \le
      K(\alpha)+K(\alpha-K(\alpha)-k)+k+\cnew
      \le
      \alpha
    \end{equation}
    for a large enough $\cref{c:C}$;
    the second inequality in \eqref{eq:alpha-bound} follows
    from the inequality $K(s)\le s/2+\cnew$ applied to
    $s:=\alpha-K(\alpha)-k$.
  \end{itemize}
  The second item characterizes the complexity of the model $Q_k$,
  $K(Q_k)\le\alpha$;
  let us check that the first item characterizes its quality.
  Since $Q_k(\{\omega\})\ge2^{k-K(\omega)}$,
  the randomness e-deficiency of $\omega$ w.r.t.\ $Q_k$ is
  \begin{align}
    -\log Q_k(\{\omega\})
    -
    K(\omega\mid Q_k)
    &\le
    \bigl(
      K(\omega)-k
    \bigr)
    -
    \bigl(
      K(\omega)-K(Q_k)
    \bigr)
    +
    \cnew\notag\\
    &=
    K(Q_k)-k+\cold
    \le
    K(k)+\cnew\clabel{c:deficiency},
    \label{eq:deficiency}
  \end{align}
  where the last inequality follows from \eqref{eq:first}.
  Since $\omega\in\NNN^{\mathrm{q}}_{\alpha,\beta}$,
  we have $\beta<K(k)+\cold\le2\log(k+1)+\cnew$,
  which implies $k>2^{\beta/2-\cnew\clabel{c:k-beta}}-1$.

  We can see that $\NNN^{\mathrm{q}}_{\alpha,\beta}$
  is covered by
  \begin{equation}\label{eq:covered}
    \{\omega\st k(\omega)>\alpha-K(\alpha)-\cref{c:C}\}
    \quad\text{and}\quad
    \{\omega\st k(\omega)>2^{\beta/2-\cref{c:k-beta}}-1\}.
  \end{equation}
  It follows immediately from the definition that
  \begin{equation}\label{eq:slice-tail}
    \m(\{\omega\st k(\omega)\ge j\})\le\cnew2^{-j}
    \quad
    \text{for all $j\in\N_0$},
  \end{equation}
  and so \eqref{eq:covered} gives
  \[
    \m(\NNN^{\mathrm{q}}_{\alpha,\beta})
    \le
    \cnew
    \left(
      2^{-\alpha+K(\alpha)+\cref{c:C}}
      +
      2^{-2^{\beta/2-\cref{c:k-beta}}}
    \right).
    \qedhere
  \]
\end{proof}

Let us check that the upper bound of Theorem~\ref{thm:quasi}
can be slightly improved to
\[
  \m(\NNN^{\mathrm{q}}_{\alpha,\beta})
  \le
  c
  \left(
    2^{-\alpha+K(\alpha)}
    +
    2^{-2^{\beta-K(\beta)-c}}
  \right),
\]
assuming $\beta\in\N_0$.
The only required change in the proof
is to deduce $k>2^{\beta-K(\beta)-\cnew}-1$
from $\beta<K(k)+\cref{c:deficiency}$.
To establish this lower bound on $k$, we argue indirectly.
Every $k'\le2^{\beta-K(\beta)-\cold}-1$
has the following self-delimiting description:
a shortest $U$-description of $\beta$ (also giving $K(\beta)$),
a shortest self-delimiting description of $\cold$,
and $\beta-K(\beta)-\cold$ bits identifying $k'$
(assuming $\cold\in\N$).
For a large enough $\cold$,
this will contradict $\beta<K(k')+\cref{c:deficiency}$.

Another proof technique that we need is a version of Shen's ``first uncovered object''.
The next lemma presents this technique both
for prefix complexity $K$ (which we need now) and
for plain complexity $C$ (which we will need later).
Remember that the busy beaver function $\BB$ is defined
by \eqref{eq:BB} in Sect.~\ref{sec:main}.
The $C$-counterpart of this function is
$\BB^C(s):=\max\{n\st C(n)\le s\}$.

\begin{lemma}\label{lem:BB}
  For $s\in\N_0$,
  set $\omega^*_s:=\BB(s)+1$ and $\omega^{*C}_s:=\BB^C(s)+1$.
  There is a constant $\cnew\clabel{c:BB}$ such that,
  for all $s$,
  \begin{itemize}
  \item
    $s<K(\omega^*_s)\le s+\cref{c:BB}$
    and every finite set $\Omega\subseteq\XXX$ such that $K(\Omega)\le s-\cref{c:BB}$
    satisfies $\max\Omega<\omega^*_s$
    (in particular, $\omega^*_s\notin\Omega$);
  \item
    $s<C(\omega^{*C}_s)\le s+\cref{c:BB}$
    and every finite set $\Omega\subseteq\XXX$ such that $C(\Omega)\le s-\cref{c:BB}$
    satisfies $\max\Omega<\omega^{*C}_s$
    (in particular, $\omega^{*C}_s\notin\Omega$).
  \end{itemize}
\end{lemma}

\begin{proof}
  The $K$- and $C$-versions are proved in the same way;
  let us only discuss the $K$-version.
  The first inequality, $K(\omega^*_s)>s$, immediately follows from the definition of $\BB$,
  and the second, $K(\omega^*_s)\le s+\cref{c:BB}$,
  follows from $K(\BB(s))\le s$ when $\BB(s)>0$
  (for $\BB(s)=0$, it suffices to make $\cref{c:BB}\ge K(1)$).
  If a nonempty finite $\Omega\subseteq\XXX$ satisfies $K(\Omega)\le s-\cref{c:BB}$
  for a large enough $\cref{c:BB}$,
  then $K(\max\Omega)\le s$,
  which implies $\max\Omega\le\BB(s)<\omega^*_s$.
\end{proof}

In particular,
all objects $\omega\ge\omega^*_{\alpha+\cref{c:BB}}$ belong to $\NNN_{\alpha,\infty}$,
and all objects $\omega\ge\omega^{*C}_{\alpha+\cref{c:BB}}$ belong to $\NNN^C_{\alpha,\infty}$
(in the notation of Sect.~\ref{sec:main}).
Shen's ``first uncovered object'' \cite[Theorem~2(b)]{Shen:1983}
is the smallest object in $\NNN_{\alpha,\infty}$ or $\NNN^C_{\alpha,\infty}$;
its modifications $\omega^*$ and $\omega^{*C}$ are more convenient here.

\begin{proof}[Proof of the lower bound of Theorem~\ref{thm:quasi}]
  If $\omega\notin\NNN^{\mathrm{q}}_{\alpha,\beta}$,
  then $\omega$ belongs to the support $\Omega$ of an elementary probability measure $Q$
  with $K(Q)\le\alpha$;
  we have $K(\Omega)\le\alpha+\cnew$.
  Therefore, $\NNN^{\mathrm{q}}_{\alpha,\infty}$ contains all objects
  that do not belong to any finite set of prefix complexity at most $\alpha+\cold$.
  By Lemma~\ref{lem:BB},
  one of those objects has prefix complexity at most $\alpha+\cnew$,
  which gives the lower bound in the theorem.
\end{proof}

\begin{proof}[Proof of Theorem~\ref{thm:unconditional}]
  The lower bound in~\eqref{eq:unconditional} is proved
  in the same way, verbatim, as the lower bound in Theorem~\ref{thm:quasi}.

  The upper bound in~\eqref{eq:unconditional} has already been proved,
  essentially, in the proof of the upper bound in Theorem~\ref{thm:quasi}.
  With the unconditional definition,
  we have, in place of \eqref{eq:deficiency},
  \begin{equation*}
    -\log Q_k(\{\omega\})
    -
    K(\omega)
    \le
    \bigl(
      K(\omega)-k
    \bigr)
    -
    K(\omega)
    =
    -k,
  \end{equation*}
  and so the covering family \eqref{eq:covered} only requires
  the set $\{\omega\st k(\omega)>\alpha-K(\alpha)-\cref{c:C}\}$.

  It remains to prove the second part of the theorem.
  We will use G\'acs's result \cite{Gacs:1983}
  essentially in the form given by Andreev \cite[Theorem~2(ii)]{Andreev:2017}:
  if $(x_n,y_n)$, $n\in\N$, are distinct pairs of nonnegative integers
  such that $x_n\le y_n$ for all $n$,
  the sequence $(x_n)$ is lower semicomputable,
  the sequence $(y_n)$ is upper semicomputable,
  and $\sum_n2^{x_n-y_n}=\infty$,
  then $\tau(\BB(y_n))\ge2^{-x_n}$ for some $n$.
  This result trivially extends to $n$ ranging over the integers $n\ge n_0$
  for some $n_0\in\N$.
  For $f$ as in the statement of the theorem and assuming $n_0\ge K(1)$,
  apply G\'acs's result to
  \[
    x_n
    :=
    (n-f(n)-1)^+
    \le
    n+\cnew
    =:
    y_n,
    \quad
    n\ge n_0,
  \]
  where $\cold$ is a constant that is large enough
  (in particular, as compared with $\cref{c:BB}$ in Lemma~\ref{lem:BB}).
  We obtain $n\ge n_0$ such that $\tau(\BB(n+\cold))\ge2^{-x_n}$.
  Since $\BB(n+\cold)\ge1$
  and $\tau(t)\le1-\m(1)<1$ for all $t\in\N$,
  we have $x_n\ne0$, which implies $x_n=n-f(n)-1$.
  By Lemma~\ref{lem:BB},
  \begin{equation*}
    \NNN^{\mathrm{qu}}_{\alpha,\beta}
    \supseteq
    \{\omega\st\omega>\mathrm{BB}(\alpha+\cold)\}
  \end{equation*}
  holds for all $\alpha\in\N_0$ and $\beta\in[0,\infty]$;
  therefore,
  \[
    \m(\NNN^{\mathrm{qu}}_{n,\beta})
    \ge
    \tau(\BB(n+\cold))
    \ge
    2^{-n+f(n)+1}
    >
    2^{-n+f(n)}
  \]
  for all $\beta\in[0,\infty]$.
  Since $n_0$ can be arbitrarily large, this completes the proof.
\end{proof}

\subsection{Proof of Theorem~\ref{thm:Q}}
\label{app:Q}

The proof of the lower bound of Theorem~\ref{thm:quasi}
relies on the support of elementary probability measures
being finite.
Computable probability measures can have the whole of $\XXX$ as their support,
of course,
but the following lemma shows that their tails
can be both light and simple.

\begin{lemma}\label{lem:cut}
  There is a constant $\cnew\clabel{c:cut}$ such that the following holds
  for all $\alpha,n\in\N_0$.
  If a probability measure $P$ on $\XXX$
  has a number $\pi$ with $K(\pi)\le\alpha$,
  there is $N\in\N$ such that
  $P(\XXX\setminus[N])<2^{-\alpha-n}$ and $K(N)\le\alpha+K(n)+\cref{c:cut}$.
\end{lemma}

The parameter~$n$ in Lemma~\ref{lem:cut} describes the trade-off
between the lightness and simplicity of tails.
Since the upper bound on $K(N)$ involves the complexity of $n$,
a judicious choice of $n$ (using the function $\mu$)
will allow a very large value of $n$ and so a very light tail.
For a zero tail, as in the case of elementary probability measures,
we have no dependence of the lower bound on $\beta$,
while now we will have light dependence on $\beta$.

\begin{proof}[Proof of Lemma~\ref{lem:cut}]
  Set $d:=\alpha-K(\pi)\ge0$, and let $\pi^*$
  be a shortest self-delimiting description of $\pi$.
  Given $\pi^*$ and $n+d$,
  we can compute $N\in\N$
  such that $P(\XXX\setminus[N])<2^{-\left|\pi^*\right|-n-d}=2^{-\alpha-n}$.
  Therefore,
  \[
    K(N)
    \le
    K(\pi)+K(n+d)+\cnew
    \le
    K(\pi)+K(n)+d+\cref{c:cut}
    =
    \alpha+K(n)+\cref{c:cut},
  \]
  the second inequality holding for a large enough $\cref{c:cut}$.
\end{proof}

Now we can prove the theorem.
Let $\cnew\clabel{c:large}$ be a large constant (see below),
let $n>2\beta+2\cref{c:large}$ satisfy $K(n)=\mu(2\beta+2\cref{c:large})$,
and let $\omega^*:=\omega^*_{\alpha'}$ be the object of Lemma~\ref{lem:BB}
for $\alpha':=\alpha+K(n)+\cref{c:large}$;
therefore, $K(\omega^*)\le\alpha'+\cref{c:BB}$
and, for all $N\ge\omega^*$, $K(N)>\alpha'$.

The object $\omega^*$ alone will witness the inequality in Theorem~\ref{thm:Q}.
To show that $\omega^*\in\NNN^{\mathrm{Q}}_{\alpha,\beta}$,
we argue indirectly.
Suppose $\omega^*$ is $(\alpha,\beta)$-Quasi-stochastic
as witnessed by $P$ and its number $\pi$.
We then have
\begin{align}
  -\log P(\{\omega^*\})
  &\le
  K(\omega^*\mid\pi)+\beta
  \le
  K(\omega^*)+\beta+\cnew
  \le
  \alpha'+\beta+\cnew\notag\\
  &=
  \alpha+K(n)+\cref{c:large}+\beta+\cold
  <
  \alpha+n;
  \label{eq:have}
\end{align}
the last inequality holds since $\beta+\cref{c:large}<n/2$
and, for a large $\cref{c:large}$ (using $n>2\cref{c:large}$),
$K(n)+\cold\le n/2$.
Now let $N$ be given by Lemma~\ref{lem:cut} for $P$, $\pi$, and $n$.
Since $P(\XXX\setminus[N])<2^{-\alpha-n}<P(\{\omega^*\})$,
we have $N\ge\omega^*$.
On the other hand, $K(N)\le\alpha+K(n)+\cref{c:cut}\le\alpha'$ for a large $\cref{c:large}$,
which contradicts $N\ge\omega^*$.
Therefore, $\omega^*\in\NNN^{\mathrm{Q}}_{\alpha,\beta}$.

Now we obtain
\[
  \m(\NNN^{\mathrm{Q}}_{\alpha,\beta})
  \ge
  \m(\omega^*)
  \ge
  \cnew 2^{-K(\omega^*)}
  \ge
  \cold 2^{-\alpha'-\cref{c:BB}}
  =
  \cold 2^{-\alpha-\mu(2\beta+2\cref{c:large})-\cref{c:large}-\cref{c:BB}}.
\]
To complete the proof of the theorem,
it remains to notice that $\mu(2\beta+2\cref{c:large})\le\mu(\beta)+\cnew$.

Apart from $\cref{c:large}$,
every constant in this proof depends only on previously defined constants.

The remark about $\NNN^{\mathrm{Qu}}_{\alpha,\beta}$
made after Theorem~\ref{thm:Q} is obvious from \eqref{eq:have}:
we can just remove the expression between the first and second ``$\le$''.

\subsection{Proofs of Theorems~\ref{thm:penalty}, \ref{thm:dichotomy}, and \ref{thm:beta}}
\label{app:stronger}

This section gives (among other things) independent derivations
of~\eqref{eq:dichotomy-C} and~\eqref{eq:dichotomy-mu} in Theorem~\ref{thm:dichotomy}
and of Theorem~\ref{thm:beta}.
Alternatively, we could have derived these statements as corollaries
of Theorems~\ref{thm:eight} and~\ref{thm:eight-set}.

\begin{proof}[Proof of Theorem~\ref{thm:penalty}]
  Let $\omega\notin\NNN_{\alpha,\beta}(d)$,
  as witnessed by an elementary probability measure $Q$
  satisfying $K(Q)\le\alpha$ and $d(\omega\mid Q)\le\beta$.
  We have
  \begin{align}
    K(\omega)
    &\le
    K(Q)-\log Q(\{\omega\})+\cnew
    \le
    \alpha+F(\omega)+d_F(\omega\mid Q)+\cold\notag\\
    &\le
    \alpha+F(\omega)+c d(\omega\mid Q)+\cnew
    \le
    \alpha+F(\omega)+c\beta+\cold,
    \label{eq:chain}
  \end{align}
  where the third ``$\le$'' follows from~\eqref{eq:penalty}.
  Therefore, $\NNN_{\alpha,\beta}(d)$ contains all $\omega$ satisfying $K(\omega)-F(\omega)>t$,
  where $t:=\lceil\alpha+c\beta+\cold\rceil$,
  and Lemma~\ref{lem:Kraft} gives $\m(\NNN_{\alpha,\beta}(d))\ge\cnew\tau(t)$.
  Applying Lemma~\ref{lem:g} completes the proof.
\end{proof}

The proof of Theorem~\ref{thm:penalty} shows
that Theorem~\ref{thm:penalty} and its $C$-version
(with $\NNN^C_{\alpha,\beta}(d)$ in place of $\NNN_{\alpha,\beta}(d)$)
can be strengthened
by replacing \eqref{eq:penalty} with
\begin{equation}\label{eq:penalty-bis}
  d_F(\omega\mid Q)
  \le
  c d(\omega\mid Q) + c K(Q) + c
\end{equation}
(the chain \eqref{eq:chain} already contains $K(Q)$,
and we can use $K(Q)\le\alpha$ again).

Let us state explicitly inequalities
that have already been used in narrower contexts
(and will be used again in the following proofs).
Let $\delta_{\omega}$ be the point mass at $\omega\in\XXX$.
The following proposition gives two extreme possibilities
for the universal measure of non-quasi-stochastic objects.
(All deficiencies considered in this paper satisfy both conditions
given in the proposition.)

\begin{proposition}\label{prop:extremes}
  Let $d$ be a deficiency.
  If $d(\omega\mid Q)=\infty$ whenever $Q(\{\omega\})=0$,
  there exists $\cnew>0$ such that
  $\m(\NNN_{\alpha,\beta}(d))\ge2^{-\alpha-\cold}$
  for all $\alpha$ and $\beta$.
  If $d(\omega\mid\delta_\omega)\le\cnew$ for all $\omega$,
  where $\cold$ is a constant,
  there exists $\cnew>0$ such that
  $\m(\NNN_{\alpha,\beta}(d))\le\cold\tau(\alpha)$
  for all $\alpha$ and all $\beta\ge\colder$.
\end{proposition}

\begin{proof}
  The lower bound is proved in the same way
  as the lower bound in Theorem~\ref{thm:quasi},
  and the upper bound is proved
  in the same way as the upper bound in~\eqref{eq:super}.
\end{proof}

\begin{proof}[Proof of Theorem~\ref{thm:dichotomy}]
  We start from the case $\sum_\omega2^{-F(\omega)}<\infty$,
  where we have $F\ge K-\cnew\clabel{c:F}$ and thus $d_F\le d_K+\cref{c:F}$.
  Therefore,
  \[
    \NNN_{\alpha,\beta}(d_F)
    \subseteq
    \NNN_{\alpha,\beta-\cref{c:F}}(d_K)
    =
    \NNN^{\mathrm{qu}}_{\alpha,\beta-\cref{c:F}}
  \]
  for all $\beta\ge\cref{c:F}$,
  and the upper bound in~\eqref{eq:dichotomy-K} (with the extra condition $\beta\ge\cref{c:F}$)
  follows from Theorem~\ref{thm:unconditional}.
  The lower bound in~\eqref{eq:dichotomy-K} follows from Proposition~\ref{prop:extremes}.

  Let us now get rid of the extra condition $\beta\ge\cref{c:F}$ in the upper bound;
  it suffices to prove that
  $\NNN_{\alpha+\cnew,0}(d_F)\subseteq\NNN_{\alpha,\cref{c:F}}(d_F)$
  for some constant $\cold$.
  To do that, take $\omega\notin\NNN_{\alpha,\cref{c:F}}(d_F)$,
  witnessed by a model $Q$ with $K(Q)\le\alpha$ and $d_F(\omega\mid Q)\le\cref{c:F}$.
  Cut the support of $Q$, in increasing order, into blocks:
  each atom of mass greater than $2^{-\cref{c:F}}$ forms a block on its own;
  the remaining atoms are grouped into maximal consecutive runs of total mass
  at most $2^{-\cref{c:F}}$.
  This gives fewer than $2^{\cref{c:F}+2}$ blocks.
  Then $Q':=Q(\cdot\mid B)$, where $B$ is the block containing $\omega$,
  will witness $\omega\notin\NNN_{\alpha+\cold,0}(d_F)$
  for a sufficiently large $\cold$.

  Now suppose $\sum_\omega2^{-F(\omega)}=\infty$.
  The lower bound in~\eqref{eq:dichotomy-C} is a special case of Theorem~\ref{thm:penalty},
  and the upper bound follows from Proposition~\ref{prop:extremes}.

  Let us prove~\eqref{eq:dichotomy-mu}.
  Combining the lower bounds in~\eqref{eq:dichotomy-C} and Proposition~\ref{prop:extremes}
  with Lemma~\ref{lem:Andreev},
  we obtain
  \begin{multline*}
    -\log\m(\NNN_{\alpha,\beta}(d_F))
    \le
    \min\bigl(\alpha,-\log\tau(\max(\alpha,\beta))\bigr)+\cnew\\
    \le
    \min\bigl(\alpha,\mu(\max(\alpha,\beta))\bigr)+\cnew
    \le
    \min(\alpha,\mu(\beta))+2\log(\alpha+2)+\cnew,
  \end{multline*}
  where the last inequality (for a large enough $\cold$) follows from
  $\mu(\max(\alpha,\beta))=\max(\mu(\alpha),\mu(\beta))\le\mu(\alpha)+\mu(\beta)$
  and $\mu(\alpha)\le K(\alpha+1)\le2\log(\alpha+2)+O(1)$.

  For the opposite inequality
  \begin{equation}\label{eq:opposite}
    -\log\m(\NNN_{\alpha,\beta}(d_F))
    \ge
    \min(\alpha,\mu(\beta))-c\log(\alpha+2)
  \end{equation}
  we use the uniform probability measures $U_{[2^b]}$, $b\in\N_0$, on $[2^b]$;
  we have $K(U_{[2^b]})\le K(b)+\cnew$,
  and $d_F(\omega\mid U_{[2^b]})\le b$ for all $\omega\in[2^b]$.
  Let $b^*$ be the largest integer $b\in[0,\beta]$ satisfying $K(b)\le\alpha-\cold$;
  if there is no such $b$, then $\alpha<K(0)+\cold$,
  and~\eqref{eq:opposite} holds for a large enough $c$.
  All $\omega\in[2^{b^*}]$ are $(\alpha,\beta)$-quasi-stochastic w.r.t.\ $d_F$,
  and so $\m(\NNN_{\alpha,\beta}(d_F))\le\tau(2^{b^*})\le\tau(b^*)$.
  By the definition of $b^*$,
  all integers $n\in(b^*,\beta]$ satisfy $K(n)>\alpha-\cold$,
  and all $n>\beta$ satisfy $K(n)\ge\mu(\beta)$;
  therefore, $\mu(b^*)\ge\min(\alpha-\cold,\mu(\beta))$.
  It remains to apply Andreev's estimate $-\log\tau\ge\mu-O(\log\mu)$
  (see~\eqref{eq:Andreev}) to $\tau(b^*)$,
  taking into account that the function $x\mapsto x-c\log(x+2)$
  is increasing for large $x$.
\end{proof}

\begin{proof}[Proof of Theorem~\ref{thm:beta}]
  The proof of~\eqref{eq:dichotomy-mu} carries over verbatim
  after the following three replacements:
  \begin{itemize}
  \item
    the lower bound in~\eqref{eq:dichotomy-C}
    is replaced by the lower bound in~\eqref{eq:super};
  \item
    the lower bound $\m(\NNN_{\alpha,\beta}(d_F))\ge2^{-\alpha-c}$
    is replaced by $\m(\NNN_{\alpha,\beta})\ge2^{-\alpha-c}$,
    which follows from Lemma~\ref{lem:BB}:
    for $\omega^*:=\omega^*_{\alpha+\cref{c:BB}}\in\NNN_{\alpha,\infty}\subseteq\NNN_{\alpha,\beta}$
    we have $K(\omega^*)\le\alpha+2\cref{c:BB}$,
    and so $\omega^*\in\NNN_{\alpha,\beta}$ and
    \[
      \m(\NNN_{\alpha,\beta})
      \ge
      \m(\omega^*)
      \ge
      \cnew 2^{-\alpha-2\cref{c:BB}};
    \]
  \item
    the uniform probability measures $U_{[2^b]}$ on $[2^b]$
    are replaced by the finite sets $[2^b]$.
    \qedhere
  \end{itemize}
\end{proof}

\subsection{Proof of Theorem~\ref{thm:eight}}
\label{app:eight}

We prove the upper and lower bounds separately.

\subsubsection*{Upper bounds}

We will use the slices $k(\omega)$ and the elementary probability measures $Q_k$
defined in the proof of the upper bound of Theorem~\ref{thm:quasi}
in Sect.~\ref{app:quasi}.
Besides~\eqref{eq:alpha-bound}, \eqref{eq:deficiency}, \eqref{eq:slice-tail},
and
\begin{equation}\label{eq:noticed}
  Q_{k(\omega)}(\{\omega\})\ge2^{k(\omega)-K(\omega)}
\end{equation}
(noticed just before~\eqref{eq:deficiency}),
we need one further property of the slices.
Since the first $k+1$ bits of $\Chaitin$ form a plain description of $Q_k$
(their number giving $k$),
we have
\begin{equation}\label{eq:C-slice}
  C(Q_k)\le k+\cnew\clabel{c:C-slice}.
\end{equation}

For a deficiency $d$ and $\beta\in[0,\infty)$, set
\[
  B_{\beta}(d)
  :=
  \{\omega\st d(\omega\mid Q_{k(\omega)})>\beta\};
\]
these are the $\omega$ for which $Q_{k(\omega)}$ does not serve as a good model.
Let $\omega\notin B_\beta(d)$.
If $k(\omega)\le\alpha-K(\alpha)-\cref{c:C}$,
then $K(Q_{k(\omega)})\le\alpha$ by~\eqref{eq:alpha-bound},
and so $\omega\notin\NNN_{\alpha,\beta}(d)$;
if $k(\omega)\le\alpha-\cref{c:C-slice}$,
then $C(Q_{k(\omega)})\le\alpha$ by~\eqref{eq:C-slice},
and so $\omega\notin\NNN^C_{\alpha,\beta}(d)$.
By~\eqref{eq:slice-tail},
there is a constant $\cnew\clabel{c:reduction}$ such that,
for all $d$, $\alpha$, and $\beta$,
\begin{align}
  \m(\NNN_{\alpha,\beta}(d))
  &\le
  \cref{c:reduction}
  \left(2^{-\alpha+K(\alpha)}+\m(B_\beta(d))\right),
  \label{eq:reduction-K}\\
  \m(\NNN^C_{\alpha,\beta}(d))
  &\le
  \cref{c:reduction}
  \left(2^{-\alpha}+\m(B_\beta(d))\right).
  \label{eq:reduction-C}
\end{align}
It remains to bound $\m(B_\beta(d))$ for our four deficiencies.

Since $d_K(\omega\mid Q_{k(\omega)})\le-k(\omega)\le0$ (see \eqref{eq:noticed}),
the set $B_\beta(d_K)$ is empty,
which gives us the upper bounds in~\eqref{eq:K-K} and~\eqref{eq:C-K}.

By~\eqref{eq:deficiency}
(which holds for every $\omega$ satisfying $k(\omega)=k$),
every $\omega\in B_\beta(\dL)$ satisfies
$\beta<K(k(\omega))+\cref{c:deficiency}\le2\log(k(\omega)+1)+\cnew$,
which implies $k(\omega)+1>2^{(\beta-\cold)/2}$;
therefore, \eqref{eq:slice-tail} gives
$\m(B_\beta(\dL))=O(2^{-2^{(\beta-\cold)/2}})$.
This gives the upper bounds in~\eqref{eq:K-e} and~\eqref{eq:C-e}.

Next let $\omega\in B_\beta(\dML)$ and $k:=k(\omega)$.
Combining~\eqref{eq:deficiency} and
\begin{equation*}
  \dL(\omega\mid Q)
  \ge
  \dML(\omega\mid Q)
  -
  2\log(\dML(\omega\mid Q)+2)
  -
  \cnew
\end{equation*}
(see the argument in \cite[Proposition~1]{Novikov:arXiv1608}),
we obtain
\[
  K(k)
  \ge
  \dL(\omega\mid Q_k)-\cnew
  \ge
  \beta-2\log(\beta+2)-\cnew.
\]
On the other hand, if $k<2^{\beta/2}$,
a shortest $U$-description of $j:=\lceil\beta/2\rceil$
followed by the $j$-bit binary representation of $k$
is a self-delimiting description of $k$,
and so
\[
  K(k)\le j+K(j)+\cnew\le\beta/2+2\log(\beta+2)+\cnew.
\]
The two bounds on $K(k)$ are incompatible when $\beta\ge\cnew\clabel{c:p-large}$
(for a large enough $\cref{c:p-large}$).
Therefore, for $\beta\ge\cref{c:p-large}$,
$B_\beta(\dML)\subseteq\{\omega\st k(\omega)\ge2^{\beta/2}\}$,
and~\eqref{eq:slice-tail} gives $\m(B_\beta(\dML))=O(2^{-2^{\beta/2}})$.
This gives the upper bounds in~\eqref{eq:K-p} and~\eqref{eq:C-p}
(for $\beta<\cref{c:p-large}$, they are trivial for a large enough $c$).

Finally,
by \eqref{eq:noticed} and $F\ge0$,
\[
  d_F(\omega\mid Q_{k(\omega)})\le K(\omega)-k(\omega)-F(\omega)\le K(\omega),
\]
and so $B_\beta(d_F)\subseteq\{\omega\st K(\omega)>\beta\}$,
whose measure is $\sigma(\beta)=O(\tau(\beta))$ by Lemma~\ref{lem:sigma-tau}.
By~\eqref{eq:reduction-K} and~\eqref{eq:reduction-C},
\[
  \m(\NNN_{\alpha,\beta}(d_F))=O(2^{-\alpha+K(\alpha)}+\tau(\beta))
  \quad\text{and}\quad
  \m(\NNN^C_{\alpha,\beta}(d_F))=O(2^{-\alpha}+\tau(\beta)).
\]
On the other hand,
since $d_F(\omega\mid\delta_\omega)=-F(\omega)\le0$,
Proposition~\ref{prop:extremes} gives $\m(\NNN_{\alpha,\beta}(d_F))=O(\tau(\alpha))$.
Since $C\le K+\cnew\clabel{c:CK}$,
we have $\NNN^C_{\alpha,\beta}(d)\subseteq\NNN_{\alpha-\cref{c:CK},\beta}(d)$
for $\alpha\ge\cref{c:CK}$,
and so Lemma~\ref{lem:g} gives $\m(\NNN^C_{\alpha,\beta}(d_F))=O(\tau(\alpha))$ as well.
Combining these bounds
and using
\begin{equation}\label{eq:min}
  \min(\tau(\alpha),x+\tau(\beta))\le x+\tau(\max(\alpha,\beta))
\end{equation}
for $x\ge0$,
we obtain the upper bounds in~\eqref{eq:K-C} and~\eqref{eq:C-C}.

\subsubsection*{Lower bounds}

All four deficiencies satisfy $d(\omega\mid Q)=\infty$ whenever $Q(\{\omega\})=0$.
Therefore, Proposition~\ref{prop:extremes}
gives the lower bound $2^{-\alpha}/c$ in~\eqref{eq:K-e}--\eqref{eq:K-C}.
Similarly, every $\omega\notin\NNN^C_{\alpha,\beta}(d)$
belongs to the support of an elementary probability measure $Q$ with $C(Q)\le\alpha$,
which is a finite set of plain complexity at most $\alpha':=\alpha+\cnew$.
By Lemma~\ref{lem:BB},
the object $\omega^*:=\omega^{*C}_{\alpha'+\cref{c:BB}}\in\NNN^C_{\alpha',\infty}$
has $C(\omega^*)$ within an additive constant of $\alpha'$,
and so it satisfies
\begin{equation}\label{eq:star}
  K(\omega^*)
  \le
  C(\omega^*)+K(C(\omega^*))+O(1)
  \le
  \alpha'+K(\alpha')+O(1)
  \le
  \alpha+K(\alpha)+O(1)
\end{equation}
(for the first inequality,
notice that a shortest self-delimiting description of $C(\omega^*)$
followed by a shortest plain description of $\omega^*$
is a self-delimiting description of $\omega^*$).
This gives the lower bound $2^{-\alpha-K(\alpha)}/c$
in~\eqref{eq:C-e}--\eqref{eq:C-C}.
Finally, the terms $\tau(\max(\alpha,\beta))$
in the lower bounds in~\eqref{eq:K-C} and~\eqref{eq:C-C}
are given by Theorem~\ref{thm:penalty}
(with $d:=d_F$, for which~\eqref{eq:penalty} is trivial)
and its $C$-version stated after it.

\subsection{Proof of Theorem~\ref{thm:eight-set}}
\label{app:eight-set}

The proof of the upper bound in Theorem~\ref{thm:eight-set}
will use the following lemma.
Let us say that $I\subseteq\N$ is an \emph{interval}
if $I=\{n\st a<n\le b\}$ for some $a,b\in\N_0$.

\begin{lemma}\label{lem:gap}
  There is a constant $\cnew\clabel{c:gap}$ such that,
  for all $s\in\N_0$ and all intervals $I\subseteq\N$
  satisfying $K(n)>s$ for all $n\in I$,
  \begin{equation}\label{eq:gap}
    \m(I)\le2^{-s+\cref{c:gap}\log(s+2)}.
  \end{equation}
\end{lemma}

\begin{proof}
  Fix a non-empty interval $I=\{n\st a<n\le b\}$
  (for $I=\emptyset$,
  \eqref{eq:gap} is trivial).
  We use the slices and the elementary probability measures $Q_k$
  defined in the proof of the upper bound of Theorem~\ref{thm:quasi}.
  It is clear that
  \begin{equation}\label{eq:gap-slices}
    \m(I)\le\cnew\clabel{c:gap-m}\sum_{k=0}^{\infty}2^{-k}Q_k(I).
  \end{equation}

  Let us estimate $Q_k(I)$ for $k\in\N_0$.
  Suppose $Q_k(I)>2^{-r}$ for $r\in\N$;
  we will find a simple $y\in I$.
  Set $G(x):=Q_k([x])$ for $x\in\N_0$ (with $G(0):=0$)
  and
  \[
    y
    :=
    \min
    \{
      x\st G(x)>\lceil2^{r+1}G(a)\rceil2^{-r-1}
    \}
  \]
  (so that $y$ is the first point at which $G$ exceeds $G(a)$ rounded upwards to $r+1$ binary places).
  It is clear that $y\in I$,
  and we can compute $y$ from $Q_k$, $r$, and $r+2$ bits;
  therefore, \eqref{eq:first} gives
  \[
    s
    <
    K(y)
    \le
    k+K(k)+r+K(r)+\cnew.
  \]
  Therefore, $Q_k(I)\le2^{-r}$
  for all $r\in\N$ satisfying $r+K(r)\le s-k-K(k)-\cold$,
  and taking the largest such $r$ gives
  \begin{equation}\label{eq:bound}
    Q_k(I)\le2^{-s+k+K(k)+2\log(s+2)+\cnew\clabel{c:Q}}
  \end{equation}
  for all $k$
  (for a large enough $\cref{c:Q}$).
  The case where such $r$ does not exist should be considered separately;
  \eqref{eq:bound} is then trivial
  (again for a large $\cref{c:Q}$).
  Plugging \eqref{eq:bound} for $k\le s$ and $Q_k(I)\le1$ for $k>s$
  into~\eqref{eq:gap-slices}
  and using $K(k)\le2\log(k+2)+O(1)$,
  we obtain
  \[
    \m(I)
    \le
    \cref{c:gap-m}
    \left(
      \sum_{k=0}^{s}2^{-s+K(k)+2\log(s+2)+\cref{c:Q}}
      +
      \sum_{k=s+1}^{\infty}2^{-k}
    \right)
    =
    O\left(2^{-s}(s+2)^5\right),
  \]
  and it remains to take a large enough $\cref{c:gap}$.
\end{proof}

Lemma~\ref{lem:gap} may be said to cover a middle ground
between \eqref{eq:Andreev} (covering intervals of the form $(a,\infty)$)
and Epstein and Levin's results
(covering arbitrary $I\subseteq\N$).

\begin{proof}[Proof of Theorem~\ref{thm:eight-set}]
  Let $d\in\{\dL,\dML,d_K,d_F\}$.
  We start from the lower bounds.
  Since $d\set\ge d$,
  we have $\NNN_{\alpha,\beta}(d)\subseteq\NNN_{\alpha,\beta}(d\set)$
  and $\NNN^C_{\alpha,\beta}(d)\subseteq\NNN^C_{\alpha,\beta}(d\set)$,
  and so Theorem~\ref{thm:eight}
  gives the terms $2^{-\alpha}$ in~\eqref{eq:K-set}
  and $2^{-\alpha-K(\alpha)}$ in~\eqref{eq:C-set},
  as well as the terms $\tau(\max(\alpha,\beta))$ for $d=d_F$.
  For $d\in\{\dL,\dML\}$,
  the terms $\tau(\max(\alpha,\beta))$ are given
  by Theorem~\ref{thm:penalty} with $F:=C$
  and its $C$-version (stated right after Theorem~\ref{thm:penalty}):
  \eqref{eq:penalty} holds in the form $d_C\le2\dL\set+O(1)$ by~\eqref{eq:to-simplify}
  and in the form $d_C\le\dML\set+O(1)$ by \eqref{eq:ML}.
  For $d=d_K$, a similar argument is applicable
  if we replace \eqref{eq:penalty} in Theorem~\ref{thm:penalty}
  by \eqref{eq:penalty-bis}:
  indeed, $d_C\le c d_K\set+c K(Q)+c$ follows from
  \[
    d_C(\omega\mid U_{\Omega})
    \le
    2\dL\set(\omega\mid U_{\Omega})+O(1)
    \le
    2d_K\set(\omega\mid U_{\Omega})+2K(\Omega)+O(1).
  \]

  Now we prove the upper bound~\eqref{eq:set-upper}.
  Since the point mass $\delta_\omega$ is the uniform probability measure on $\{\omega\}$,
  we have $d\set(\omega\mid\delta_\omega)=d(\omega\mid\delta_\omega)\le0$ for all $\omega$,
  and so Proposition~\ref{prop:extremes} gives $\m(\NNN_{\alpha,\beta}(d\set))=O(\tau(\alpha))$;
  as in the proof of Theorem~\ref{thm:eight},
  this implies $\m(\NNN^C_{\alpha,\beta}(d\set))=O(\tau(\alpha))$.
  On the other hand,
  for all $b\in\N_0$,
  $K(U_{[2^b]})\le K(b)+\cnew\clabel{c:2b-K}$
  and $d\set(\omega\mid U_{[2^b]})\le b+\cnew\clabel{c:2b}$ for all $\omega\in[2^b]$
  (e.g., $\dML\set(\omega\mid U_{[2^b]})=b-C(\omega\mid[2^b])+O(1)$).
  Let $b^*$ be the largest $b\in\N_0$
  satisfying $b\le\beta-\cref{c:2b}$ and $K(b)\le\alpha-\cref{c:2b-K}$
  (if there is no such $b$, then $\beta<\cref{c:2b}$ or $\alpha<K(0)+\cref{c:2b-K}$,
  and~\eqref{eq:set-upper} follows from the bound $O(\tau(\alpha))$
  or is trivial for a large enough $c$).
  All $\omega\in[2^{b^*}]$ are $(\alpha,\beta)$-quasi-stochastic w.r.t.\ $d\set$
  (with $U_{[2^{b^*}]}$ as the model),
  and so
  \[
    \m(\NNN_{\alpha,\beta}(d\set))
    \le
    \tau(2^{b^*})
    \le
    \tau(b^*)
    =
    \m(I)+\tau(\beta-\cref{c:2b}),
  \]
  where $I:=\{n\st b^*<n\le\beta-\cref{c:2b}\}$.
  By the definition of $b^*$,
  all $n\in I$ satisfy $K(n)>\alpha-\cref{c:2b-K}$,
  and so Lemma~\ref{lem:gap} gives $\m(I)\le2^{-\alpha+\cnew\log(\alpha+2)}$
  for a large enough $\cold$;
  besides, $\tau(\beta-\cref{c:2b})=O(\tau(\beta))$ by Lemma~\ref{lem:g}.
  Since $\NNN^C_{\alpha,\beta}(d\set)\subseteq\NNN_{\alpha-\cref{c:CK},\beta}(d\set)$
  (see the proof of Theorem~\ref{thm:eight}),
  the resulting bound
  \[
    \m(\NNN_{\alpha,\beta}(d\set))\le2^{-\alpha+\cold\log(\alpha+2)}+O(\tau(\beta))
  \]
  also holds for $\m(\NNN^C_{\alpha,\beta}(d\set))$
  (with a larger $\cold$).
  Combining these bounds with the earlier bound $O(\tau(\alpha))$
  and again using~\eqref{eq:min},
  we obtain~\eqref{eq:set-upper}.
\end{proof}

\subsection{Proof of Theorem~\ref{thm:eight-high}}
\label{app:eight-high}

Let $B$ be one of the sets $\{\omega\st K(\omega)>k\}$ and $\{\omega\st C(\omega)>k\}$,
and let $\NNN$ be one of the sets covered by Theorems~\ref{thm:eight} and~\ref{thm:eight-set}.
We start from the upper bounds.
For $B=\{\omega\st K(\omega)>k\}$,
we have $\m(B)=\sigma(k)=O(\tau(k))$ by~\eqref{eq:tau-sigma}.
In combination with Lemma~\ref{lem:g},
this also gives $\m(B)=O(\tau(k))$ for $B=\{\omega\st C(\omega)>k\}$.
Therefore, each upper bound $\m(\NNN)\le cU$ implies
\[
  \m(\NNN\cap B)
  \le
  \min\bigl(\m(\NNN),\m(B)\bigr)
  =
  O\bigl(\min(U,\tau(k))\bigr).
\]

As for the lower bounds,
we first unify our discussion of different ways of measuring the complexity of models
by defining the \emph{floor} to be $2^{-\alpha}$ for the sets $\NNN_{\alpha,\beta}(\cdot)$
and to be $2^{-\alpha-K(\alpha)}$ for the sets $\NNN^C_{\alpha,\beta}(\cdot)$.
In each lower bound $\m(\NNN)\ge L/c$ in Theorems~\ref{thm:eight} and~\ref{thm:eight-set},
$L$ consists of one or two addends:
the floor and, in 10 of 16 cases, $\tau(\max(\alpha,\beta))$.
Since $\min(x+y,z)\le\min(x,z)+\min(y,z)$ whenever $x,y,z\ge0$,
it suffices to prove $\m(\NNN\cap B)\ge\min(x,\tau(k))/c$
for each addend of $L$ separately.

First we deal with the floors.
Let $\NNN$ be $\NNN_{\alpha,\beta}(d)$ or $\NNN_{\alpha,\beta}(d\set)$,
where $d\in\{\dL,\dML,d_K,d_F\}$
(the proof for the $\NNN^C$ versions will be given later,
with some preparatory work for $\NNN^C$ already done in the proof for $\NNN$).
The support $\Omega$ of an elementary probability measure $Q$
satisfies $K(\Omega)\le K(Q)+\cnew$,
and so, by Lemma~\ref{lem:BB} with $s:=\alpha+\cnew$ for a large $\cold$,
no object in the tail
\[
  T:=\{\omega\st\omega>\BB(s)\}
\]
belongs to the support of an elementary probability measure $Q$ with $K(Q)\le\alpha$.
Since $d(\omega\mid Q)=d\set(\omega\mid Q)=\infty$ whenever $Q(\{\omega\})=0$,
we have $T\subseteq\NNN$.

For $n\in\N$, let $\kappa_n$ be the uniform probability measure
on the $2^n$ numbers $n2^n,\dots,(n+1)2^n-1$.
Since at most $2^{n-1}$ objects $\omega$ satisfy $K(\omega)\le n-1$
and fewer than $2^{n-1}$ objects satisfy $C(\omega)\le n-2$,
$\kappa_n$ gives probability at least $1/2$
both to $\{\omega\st K(\omega)\ge n\}$ and to $\{\omega\st C(\omega)\ge n-1\}$.
For $n\in T$, $\kappa_n$ is concentrated on $T$,
and so $\kappa_n(T\cap B)\ge1/2$ whenever $n>\max(\BB(s),k+1)$,
and Lemma~\ref{lem:kernel} gives
\begin{wideformula}\label{eq:floor}
  \m(\NNN\cap B)
  \ge
  \m(T\cap B)
  \ge
  \frac{c_{\kappa}}{2}
  \tau\bigl(\max(\BB(s),k+1)\bigr)\NarrowLineBreak
  =
  \frac{c_{\kappa}}{2}
  \min\bigl(\tau(\BB(s)),\tau(k+1)\bigr).
\end{wideformula}
It remains to notice that
$\tau(\BB(s))\ge\m(\omega^*_s)\ge2^{-s-\cnew}\ge2^{-\alpha-\cnew}$
by Lemma~\ref{lem:BB}
and that $\tau(k+1)\ge\cnew\tau(k)$ by Lemma~\ref{lem:g}.

Now let us deal with $\NNN^C_{\alpha,\beta}(d)$ and $\NNN^C_{\alpha,\beta}(d\set)$.
The argument is the same,
except that now the support $\Omega$ of $Q$ satisfies $C(\Omega)\le C(Q)+\cnew\clabel{c:supp-C}$,
and so, by Lemma~\ref{lem:BB}, $T$ can be replaced by $\{\omega\st\omega>\BB^C(s')\}$,
where $s':=\alpha'+\cref{c:BB}$ and $\alpha':=\alpha+\cref{c:supp-C}$.
In~\eqref{eq:floor}, $\BB(s)$ is then replaced by $\BB^C(s')$,
and it remains to notice that
\[
  \tau(\BB^C(s'))
  \ge
  \m(\omega^{*C}_{s'})
  \ge
  \cnew2^{-\alpha-K(\alpha)},
\]
since
\[
  K(\omega^{*C}_{s'})\le\alpha'+K(\alpha')+O(1)\le\alpha+K(\alpha)+O(1)
\]
(see \eqref{eq:star} in the proof of the lower bounds in Theorem~\ref{thm:eight}).

It remains to deal with the addend $\tau(\max(\alpha,\beta))$,
which appears in the lower bounds for $d_F$ in Theorem~\ref{thm:eight}
and for all four $d\set$ in Theorem~\ref{thm:eight-set}.
There it was obtained from Theorem~\ref{thm:penalty} or its $C$-version,
with \eqref{eq:penalty} or \eqref{eq:penalty-bis},
applied to some $F$ (perhaps $F:=C$).
The chain~\eqref{eq:chain},
combined with $K(Q)\le2C(Q)+O(1)$ for the $C$-versions,
shows that, in all these cases,
$
  \NNN
  \supseteq
  \{\omega\st K(\omega)-F(\omega)>t\}
$,
where $t:=\left\lceil\cnew(\alpha+\beta+1)\right\rceil$
for a large enough $\cold$.
Set $k':=k$ if $B=\{\omega\st K(\omega)>k\}$
and $k':=k+\lceil2\log(k+2)\rceil+\cnew$ if $B=\{\omega\st C(\omega)>k\}$;
in the latter case, $K(\omega)>k'$ implies $C(\omega)>k$,
provided $\cold$ is large enough.
Since $F\ge0$,
\[
  \{\omega\st K(\omega)-F(\omega)>\max(t,k')\}
  \subseteq
  \NNN\cap\{\omega\st K(\omega)>k'\}
  \subseteq
  \NNN\cap B,
\]
and Lemma~\ref{lem:Kraft} gives
\[
  \m(\NNN\cap B)
  \ge
  \cnew\tau\bigl(\max(t,k')\bigr)
  =
  \cold\min\bigl(\tau(t),\tau(k')\bigr).
\]
Finally, Lemma~\ref{lem:g} gives $\tau(k')\ge\cnew\tau(k)$
and $\tau(t)\ge\cnew\tau(\max(\alpha,\beta))$
(remember that, according to the definition,
$\tau(x)=\tau(\lfloor x\rfloor)$ for $x\ge0$).

\section{Kolmogorov's questions}
\label{app:Kolmogorov}

This appendix gives the historical motivation for this paper.
In the next paragraph I reproduce Kolmogorov's questions
posed in his talk on 26 November 1981
and briefly recorded by Shen \cite[note~5.12]{Semenov/etal:2024-full}.

A finite object $\omega$ is $(\alpha,\beta)$-stochastic
if there exists a finite set $\Omega$ such that
\begin{itemize}
\item
  $\omega\in\Omega$;
\item
  $K(\Omega)\le\alpha$;
\item
  $K(\omega) \ge \log\left|\Omega\right| - \beta$.
\end{itemize}
Let $N(k,\alpha,\beta)$ be the number of $(\alpha,\beta)$-stochastic objects of complexity at most $k$.
Typical cases are $\alpha,\beta=\epsilon k$ and $\alpha,\beta=\sqrt{k}$.
What can we say about the asymptotics of $N(k,\alpha,\beta)/2^k$ as $k\to\infty$?
Is it true that $N(k,\alpha,\beta)/2^k\to0$, if $k$ increases and $\alpha/k,\beta/k\to0$?

It is very likely that by $K$ Kolmogorov meant plain complexity,
nowadays usually denoted by $C$;
it appears that he never used prefix complexity.
It is also likely that there is a typo in the definition of $N(k,\alpha,\beta)$
and in fact it means the number of $(\alpha,\beta)$-nonstochastic objects of complexity at most $k$.
Therefore, in this appendix we set
\begin{equation*}
  N(k,\alpha,\beta)
  :=
  \left|
    D_k \setminus \SSS_{\alpha,\beta}
  \right|
  =
  \left|
    D_k \cap \NNN^{C\mathrm{u}}_{\alpha,\beta}
  \right|,
\end{equation*}
where $\NNN^{C\mathrm{u}}_{\alpha,\beta}$ was introduced in Sect.~\ref{sec:stronger} and
\begin{align}
  D_k
  &:=
  \{\omega\in\XXX\st C(\omega)\le k\},
  \label{eq:D_k}\\
  \SSS_{\alpha,\beta}
  &:=
  \bigl\{
    \omega\in\XXX
    \st
    \exists\Omega\subseteq\XXX:
    \left|\Omega\right|<\infty,
    \omega\in\Omega,
    C(\Omega)\le\alpha,
    C(\omega)\ge\log\left|\Omega\right|-\beta
  \bigr\}.
  \notag
\end{align}
Here we refer only to the elements of $\SSS_{\alpha,\beta}$
as $(\alpha,\beta)$-stochastic (unlike in the rest of the paper,
where stochasticity was also understood in other related senses).

\begin{remark}
  Reading ``stochastic'' literally,
  i.e., replacing $N(k,\alpha,\beta)$ with
  \[
    \left|D_k\cap\SSS_{\alpha,\beta}\right|
    =
    \left|D_k\right|-N(k,\alpha,\beta),
  \]
  makes both of Kolmogorov's questions degenerate.
  Let $c$ be a constant such that $C(\omega)\le n+c$ for all $n$ and all $\omega\in[2^n]$.
  For $n:=k-c$, the set $\Omega:=[2^n]$ is a subset of $D_k$ with $C(\Omega)\le\log k+O(1)$,
  and more than half of its elements satisfy $C(\omega)\ge n-1=\log\left|\Omega\right|-1$.
  Hence $\left|D_k\cap\SSS_{\alpha,\beta}\right|\ge2^{k-c-1}$
  as soon as $\beta\ge1$ and $\alpha\ge\log k+O(1)$,
  and the answer to the last question would be negative
  already in the typical case $\alpha=\beta=\sqrt{k}$.
  Moreover, by \eqref{eq:Kolmogorov-1} below,
  in both typical cases $\left|D_k\cap\SSS_{\alpha,\beta}\right|/2^k$
  differs by $o(1)$ from $\left|D_k\right|/2^k\in[2^{-c},2)$,
  which involves neither $\alpha$ nor $\beta$;
  so the first question would not be about stochasticity at all.
  With ``nonstochastic'', both questions become substantive:
  see Theorem~\ref{thm:Kolmogorov} and Corollary~\ref{cor:Kolmogorov}.

  Two other facts point the same way.
  Shen's 1983 paper, presented to \emph{Doklady} by Kolmogorov himself,
  bounds the number of nonstochastic objects \cite[Theorem~3]{Shen:1983},
  and V'yugin, another participant in the seminar,
  recalls Kolmogorov asking in the talk
  whether nonstochastic strings exist and, if so,
  how many of them there are among the strings of a given length
  \cite[Sect.~2]{Vyugin:2025-full}.
  (Another discrepancy between Shen's record and V'yugin's recollection
  is that V'yugin mentions lengths of strings
  while Shen's record is structure-free,
  but contemporaneous records are usually more reliable.)
  Slips in the record are not surprising:
  its definition of stochasticity contains another one,
  $K(\Omega)$ in place of $\log\left|\Omega\right|$,
  flagged by the authors \cite[note~5.12]{Semenov/etal:2024-full}
  and corrected above.
\end{remark}

Kolmogorov used the condition $C(\omega)\ge\log\left|\Omega\right|-\beta$
rather than the more standard $C(\omega\mid\Omega)\ge\log\left|\Omega\right|-\beta$
or $K(\omega\mid\Omega)\ge\log\left|\Omega\right|-\beta$,
which were the principal ones used in the main paper.
The suggestion to measure the randomness deficiency of $\omega$
as an element of $\Omega$ by
$\log\left|\Omega\right|-C(\omega\mid\Omega)$
rather than
$\log\left|\Omega\right|-C(\omega)$
is due to Per Martin-L\"of
(personal communication; see also \cite[Theorem on p.~607]{Martin-Lof:1966}),
and it leads to an elegant theory:
see, e.g., Propositions~11 and~13 in the summary \cite[Appendix~1]{Vovk:2025ML}.

\subsection{Some answers}

The following result answers Kolmogorov's specific questions;
its core part is given as \eqref{eq:Kolmogorov-main}
after its statement.

\begin{theorem}\label{thm:Kolmogorov}
  There is a positive constant $c$ such that,
  for all $k\in\N$, $\alpha\in\N_0$, and $\beta\in[0,\infty]$,
  \begin{equation}\label{eq:Kolmogorov-1}
    N(k,\alpha,\beta)
    \le
    2^{k-\alpha+c\log(k+1)}
  \end{equation}
  and, provided $\alpha+c\le k$,
  \begin{equation}\label{eq:Kolmogorov-2}
    N(k,\alpha,\beta)
    \ge
    2^{k-\alpha-K(\alpha)-c}.
  \end{equation}
\end{theorem}

\noindent
In particular,
\begin{equation}\label{eq:Kolmogorov-main}
  N(k,\alpha,\beta)
  =
  2^{k-\alpha+O(\log(k+1))}
\end{equation}
uniformly in $\beta\in[0,\infty]$
whenever $\alpha+c\le k$.
The form of~\eqref{eq:Kolmogorov-main} is familiar:
the fraction of nonstochastic objects
among the objects of complexity at most $k$
is about $2^{-\alpha}$.

The theorem will be proved in the next section,
and now we will use it to answer Kolmogorov's questions.
The following corollary uses $N(k,\alpha,\beta)=N(k,\lfloor\alpha\rfloor,\beta)$
for $\alpha$ that are not necessarily integer.

\begin{corollary}\label{cor:Kolmogorov}
  For Kolmogorov's two typical cases $\alpha,\beta=\epsilon k$
  (with $\epsilon\in(0,1)$ fixed)
  and $\alpha,\beta=\sqrt{k}$
  we have, for large $k$,
  \[
    \frac{N(k,\epsilon k,\epsilon k)}{2^k}=2^{-\epsilon k+O(\log k)} ,
    \qquad
    \frac{N(k,\sqrt k,\sqrt k)}{2^k}=2^{-\sqrt k+O(\log k)} ,
  \]
  so in both cases the ratio tends to $0$ faster than any power of $k$.
\end{corollary}

\begin{proof}
  In both cases $\alpha+c\le k$ for large $k$,
  so \eqref{eq:Kolmogorov-main} applies.
\end{proof}

Kolmogorov's last question appears to contain another slip.
A negative answer is provided already by a fixed (arbitrarily large) $\alpha$.
This can be seen both from \eqref{eq:Kolmogorov-2} and directly,
without using the theorem.
Indeed, for a fixed $\alpha$ the family $\{\Omega\st C(\Omega)\le\alpha\}$ is finite
(of size at most $2^{\alpha+1}$);
let $F_\alpha$ be its union, a finite set.
Every $(\alpha,\beta)$-stochastic object is in $F_\alpha$,
whatever $\beta$ may be,
so
$
  N(k,\alpha,\beta)
  \ge
  \left|D_k\right|-\left|F_\alpha\right|
  \ge
  2^{k-O(1)}-O(1)
$
and, therefore, 
\[
  \liminf_{k\to\infty}\frac{N(k,\alpha,\beta)}{2^k}
  \ge
  c
\]
for a positive constant $c$.
Probably Kolmogorov meant the question
``when is it true that $N(k,\alpha,\beta)/2^k\to0$?''
According to \eqref{eq:Kolmogorov-1},
``$\log k=o(\alpha)$'' is a sufficient condition.
This condition, however, is far from being necessary
(e.g., by \eqref{eq:Kolmogorov-1}, it suffices to require
$\alpha-c\log k\to\infty$).

\subsection{Proof of Theorem~\ref{thm:Kolmogorov}}

We prove the two parts of Theorem~\ref{thm:Kolmogorov},
\eqref{eq:Kolmogorov-1} and~\eqref{eq:Kolmogorov-2},
separately, starting with \eqref{eq:Kolmogorov-1}.
Since $\SSS_{\alpha,\beta}$ is increasing in (both $\alpha$ and) $\beta$,
$N(k,\alpha,\beta)$ is decreasing in $\beta$.
Therefore, it suffices to prove \eqref{eq:Kolmogorov-1} for $\beta=0$
and \eqref{eq:Kolmogorov-2} for $\beta=\infty$.

\begin{proof}[Proof of \eqref{eq:Kolmogorov-1}]
  This proof illustrates the technique of ``bounded complexity lists''
  in a simple setting.
  Fix a jointly computable way of enumerating all elements of $D_k$
  (defined by \eqref{eq:D_k})
  given $k$
  (running the optimal plain decoder on all inputs of length at most $k$ in parallel
  and listing the outputs in the order in which they first appear).
  Cutting the enumeration into consecutive blocks of length $2^{k-a}$
  and taking the block containing $\omega$ means, intuitively,
  spending about $a$ bits on the block index
  and leaving $k-a$ bits of noise inside the block.

  The idea behind the technique is that the completed blocks
  consist, roughly, of stochastic objects.
  Namely, there is a constant $\cnew$\clabel{c:blocks}
  such that the following holds
  for each $k\in\N_0$ (the \emph{level})
  and each $a\in\{0,\dots,k\}$ (the \emph{parameter}).
  Let $\omega\in D_k$ lie in a complete block of size $2^{k-a}$
  (i.e., a block that the enumeration of $D_k$ lists in full).
  Then there is a finite $\Omega\ni\omega$ with
  \begin{equation}\label{eq:blocks}
    \log\left|\Omega\right|=k-a,
    \qquad
    C(\Omega)
    \le
    a+K(k)+\cref{c:blocks}.
  \end{equation}
  Indeed, let $\Omega$ be the block containing $\omega$
  (so that the first equality in \eqref{eq:blocks}
  follows from $\left|\Omega\right|=2^{k-a}$).
  As the block is complete,
  $\Omega$ is computable from $k$, $a$,
  and the ordinal number $i<2^{a+1}$ of the block.
  Its description may consist of a self-delimiting description of $k$
  followed by exactly $a+1$ further bits for $i$
  (those further bits also determining $a$).

  Choose $k\in\N$ and set
  $a:=\alpha-\left\lfloor\cnew\log(k+1)\right\rfloor$\clabel{c:a}
  for a large $\cref{c:a}$ as compared with $\cref{c:blocks}$;
  make $c$ (in Theorem~\ref{thm:Kolmogorov}) large as compared with $\cref{c:a}$.
  If $a<0$ then $\alpha<\left\lfloor\cref{c:a}\log(k+1)\right\rfloor$,
  and the asserted bound is at least
  $2^{k+\log(k+1)}\ge2^{k+1}>\left|D_k\right|$,
  so there is nothing to prove;
  assume $a\ge0$.

  Let $\omega\in D_k$ and $k':=C(\omega)\le k$.
  If $k'<a$, apply~\eqref{eq:blocks} with the level $k'$ and the parameter $k'$:
  the block size is $1$, so $\Omega=\{\omega\}$,
  \[
    C(\Omega)
    \le
    k'+K(k')+\cref{c:blocks}
    \le
    \alpha,
  \]
  and $\log\left|\Omega\right|=0\le C(\omega)$;
  therefore, $\omega$ is $(\alpha,0)$-stochastic.
  If $k'\ge a$ and the block of $\omega$ at level $k'$ with parameter $a$ is complete,
  \eqref{eq:blocks} gives
  \[
    C(\Omega)
    \le
    a+K(k')+\cref{c:blocks}
    \le
    \alpha
  \]
  and $\log\left|\Omega\right|=k'-a\le k'=C(\omega)$,
  so again $\omega$ is $(\alpha,0)$-stochastic.

  Hence every $(\alpha,0)$-nonstochastic $\omega\in D_k$
  satisfies $k'=C(\omega)\ge a$ and lies in the last, incomplete block at its own level $k'$,
  and there are fewer than $2^{k'-a}$ such $\omega$.
  Summing over $k'\in\{a,\dots,k\}$,
  \[
    N(k,\alpha,0)
    \le
    \sum_{k'=a}^{k}2^{k'-a}
    <
    2^{k-a+1}
    =
    2^{k-\alpha+\left\lfloor\cref{c:a}\log(k+1)\right\rfloor+1}
    \le
    2^{k-\alpha+c\log(k+1)}.
    \qedhere
  \]
\end{proof}

\begin{proof}[Proof of \eqref{eq:Kolmogorov-2}]
  Set $\omega^*:=\omega^{*C}_{\alpha+\cref{c:BB}}$.
  By Lemma~\ref{lem:BB},
  all objects $\omega\ge\omega^*$ lie outside $\SSS_{\alpha,\infty}$,
  and $\alpha+\cref{c:BB}<C(\omega^*)\le\alpha+2\cref{c:BB}$.
  We can see that $C(\omega^*)$ differs from $\alpha$ by at most a constant,
  which implies
  \begin{equation}\label{eq:K-star}
    K(\omega^*)
    \le
    C(\omega^*)+K(C(\omega^*))+\cnew
    \le
    \alpha+K(\alpha)+\cnew
  \end{equation}
  (as in \eqref{eq:star}).

  Let $\Delta:=k-\alpha-K(\alpha)-\cnew\clabel{c:Delta}$,
  where $\cref{c:Delta}$ is a large integer (see below).
  If $\Delta<0$,
  it suffices to notice that $\omega^*\in D_k\setminus\SSS_{\alpha,\infty}$,
  since $C(\omega^*)\le\alpha+2\cref{c:BB}\le k$
  (using the assumption $\alpha+c\le k$ for $c\ge2\cref{c:BB}$);
  therefore, $N(k,\alpha,\infty)\ge1>2^{\Delta}$.
  Suppose now that $\Delta\ge0$,
  and consider the $2^{\Delta}$ objects $\omega^*+i$, $i\in\{0,\dots,2^{\Delta}-1\}$,
  all of which lie outside $\SSS_{\alpha,\infty}$.
  Since a shortest self-delimiting description of $\omega^*$
  followed by the binary representation of $i$
  padded with leading 0s to exactly $\Delta$ bits
  is a plain description of $\omega^*+i$,
  \eqref{eq:K-star} gives
  \begin{equation*}
    C(\omega^*+i)
    \le
    K(\omega^*)+\Delta+\cnew
    \le
    \alpha+K(\alpha)+\Delta+\cnew
    =
    k-\cref{c:Delta}+\cold
    \le
    k,
  \end{equation*}
  provided $\cref{c:Delta}$ is large enough.
  Therefore, all objects $\omega^*+i$ belong to $D_k\setminus\SSS_{\alpha,\infty}$,
  and so we have $N(k,\alpha,\infty)\ge2^{\Delta}$.
\end{proof}

The ``in particular'' bit after the theorem
follows from $K(\alpha)\le2\log(\alpha+2)+\cnew\le2\log(k+1)+\cnew$.
\end{document}